\RequirePackage{etex}
\RequirePackage{ifpdf}
\ifpdf
\documentclass[12pt]{amsart}

\usepackage{iftex}
\ifLuaTeX
\RequirePackage{luatex85}
\fi

\usepackage[foot]{amsaddr}
\usepackage[nocompress,noadjust]{cite}

\usepackage{tikz}
\usepackage{tikz-cd}
\usetikzlibrary{cd}
\usetikzlibrary{positioning, arrows.meta, matrix}
\usetikzlibrary{decorations.pathreplacing}

\usepackage{lipsum}
\usepackage{adjustbox}

\usepackage{fancybox}

\usepackage{listings}
\usepackage{docmute}
\usepackage{amssymb, enumerate, calligra}
\usepackage{amsthm}
\usepackage{amsmath}
\usepackage{amsxtra}
\usepackage{latexsym}
\usepackage{mathrsfs}
\usepackage[all,cmtip]{xy}
\usepackage[all]{xy}
\usepackage{enumitem}
\usepackage{xcolor}
\usepackage{graphicx}
\usepackage{comment}
\usepackage{mathabx,epsfig}
\usepackage{stmaryrd}

\usepackage[hypertexnames=false]{hyperref}
\usepackage[nameinlink]{cleveref}
\usepackage{autonum}

\usepackage{mathtools}

\makeatletter
\renewcommand\th@plain{\slshape}
\makeatother
\newtheoremstyle{plain}
 {2mm}
 {2mm}
 {\slshape}
 {}
 {\bfseries}
 {.}
 {.5em}
 {}
\theoremstyle{plain}
\newtheorem{theorem}{Theorem}[section]
\newtheorem{corollary}[theorem]{Corollary}
\newtheorem{lemma}[theorem]{Lemma}
\newtheorem{proposition}[theorem]{Proposition}
\newtheorem{claim}[theorem]{Claim}
\newtheorem*{claim*}{Claim}
\newtheorem{conjecture}[theorem]{Conjecture}

\newtheoremstyle{definition}
 {2mm}
 {2mm}
 {\normalfont}
 {}
 {\bfseries}
 {.}
 {.5em}
 {}
\theoremstyle{definition}

\newtheorem{definition}[theorem]{Definition}

\newtheorem{remark}[theorem]{Remark}
\newtheorem{example}[theorem]{Example}
\newtheorem*{acknowledgements}{Acknowledgements}

\crefname{section}{Section}{Sections}

\crefname{theorem}{Theorem}{Theorems}
\crefname{corollary}{Corollary}{Corollaries}
\crefname{lemma}{Lemma}{Lemmas}
\crefname{lemma}{Lemma}{Lemmas}
\crefname{proposition}{Proposition}{Propositions}
\crefname{claim}{Claim}{Claims}
\crefname{definition}{Definition}{Definitions}
\crefname{notation}{Notation}{Notations}
\crefname{problem}{Problem}{Problems}
\crefname{note}{Note}{Notes}
\crefname{remark}{Remark}{Remarks}
\crefname{example}{Example}{Examples}
\crefname{conjecture}{Conjecture}{Conjectures}
\crefname{question}{Question}{Questions}
\crefname{mainthm}{Theorem}{Theorems}
\crefname{mainprop}{Proposition}{Propositions}

\crefname{enumi}{}{}
\crefname{enumii}{}{}
\crefname{enumiii}{}{}

\crefformat{equation}{{\upshape(#2#1#3)}}

\numberwithin{equation}{section}

\def\Q{{\mathbb Q}}
\def\R{{\mathbb R}}
\def\Z{{\mathbb Z}}
\def\P{{\mathbb P}}
\def\A{{\mathbb A}}

\def\QQ{\overline{\mathbb Q}}

\def\O{{ \mathcal{O}}}
\def\m{{ \mathfrak{m}}}

\def\I{{ \mathcal{I}}}

\def\L{{ \mathcal{L}}}

\DeclareMathOperator{\pr}{pr}
\DeclareMathOperator{\Pic}{Pic}

\DeclareMathOperator{\Spec}{Spec}

\newcommand{\g}{\gamma}
\renewcommand{\d}{\delta}
\newcommand{\e}{\varepsilon}
\newcommand{\f}{\varphi}

\renewcommand{\l}{\lambda}
\renewcommand{\k}{\kappa}

\newcommand{\G}{\Gamma}
\renewcommand{\L}{\Lambda}

\allowdisplaybreaks

\begin{document}

	\title[]
	{Existence of arithmetic degrees for generic orbits and dynamical Lang-Siegel problem} 
	
	\author[Yohsuke Matsuzawa]{Yohsuke Matsuzawa}
	\address{Department of Mathematics, Graduate School of Science, Osaka Metropolitan University, 3-3-138, Sugimoto, Sumiyoshi, Osaka, 558-8585, Japan}
	\email{matsuzaway@omu.ac.jp}

	\begin{abstract} 
		We prove the existence of the arithmetic degree for dominant rational self-maps at any point whose orbit is generic.
		As a corollary, we prove the same existence for \'etale morphisms on quasi-projective varieties and any points on it.
		We apply the proof of this fact to dynamical Lang-Siegel problem. Namely, we prove that local height function associated with
		zero-dimensional subscheme grows slowly along orbits of a rational map under reasonable assumption.
		Also if local height function associated with any proper closed subscheme grows fast
		on a subset of an orbit of a self-morphism, we prove that such subset has Banach density zero under some assumptions.
	\end{abstract}

	\subjclass[2020]{Primary 37P15; 
		Secondary 
		37P55 
	}	
	
	\keywords{Arithmetic dynamics, Arithmetic degree, Dynamical Lang-Siegel problem, Local height}

	\maketitle
	\tableofcontents

\section{Introduction}\label{sec:intro}

For a dominant rational map $f \colon X \dashrightarrow X$ on a projective variety defined over $\QQ$,
\emph{arithmetic degree} is one of the most fundamental quantity that measure the arithmetic complexity 
of orbits.
It is defined using height function $h_{H}$ associated with an ample divisor on $X$ as follows:
for a point $x \in X(\QQ)$ whose $f$-orbit is well-defined, i.e.
\begin{align}
f^{n}(x) \notin I_{f} \quad n \geq 0
\end{align}
where $I_{f}$ is the indeterminacy locus of $f$,
\begin{align}
\alpha_{f}(x) := \lim_{n \to \infty} \max\{1, h_{H}(f^{n}(x))\}^{ \frac{1}{n}}
\end{align}
is called the arithmetic degree of $f$ at $x$, provided the limit exists.
It can be easily checked that this limit is independent of the choice of $H$ and $h_{H}$.
Also, we take $\max\{1, \}$ to take care the case $h_{H}(f^{n}(x))$ is bounded, but as soon as the orbit 
is infinite, Northcott theorem says $h_{H}(f^{n}(x))$ goes to $\infty$ and hence we can get rid of $\max\{1, \}$.

It is conjectured that the arithmetic degree of Zariski dense orbits is equal to another dynamical invariant
$d_{1}(f)$, the first dynamical degree (see for example \cite{DS05,Da20,Tr20} for the definition and basic properties of dynamical degree).
\begin{conjecture}[Kawaguchi-Silverman conjecture {\cite{Sil12,KS16b}}]
In the above situation,
$ \alpha_{f}(x)$ exists (i.e.\ the limit exists) and if the orbit $O_{f}(x) = \{ x, f(x), f^{2}(x), \dots\}$
is Zariski dense in $X$, then $ \alpha_{f}(x) = d_{1}(f)$.
\end{conjecture}

We refer \cite{matsuzawa2023recent} for introduction and recent advances on this conjecture.

The first part of the conjecture, that is, the existence of the limit is proven for 
self-morphisms on projective varieties \cite{KS16a}, but it is open for rational maps.
In this paper, we prove that the limit exists for $x$ whose orbit is \emph{generic}
in the following sense.

\begin{definition}
A subset $O \subset X(\QQ)$ is called generic if it is infinite and $O \cap Z$ is finite for 
every proper closed subset $Z \subset X$.
\end{definition}

\begin{theorem}\label{intro:thm-ad-exists}\footnote{When the author wrote the first version of this paper, Junyi Xie informed me that he was also aware of this theorem. The author appreciates his permission to include this theorem in this paper.}
Let $f \colon X \dashrightarrow X$ be a dominant rational map on a projective variety $X$ defined over $\QQ$.
Let $H$ be an ample divisor on $X$ and fix Weil height function $h_{H}$ associated with $H$.
Let $x \in X_{f}(\QQ)$.
If $O_{f}(x)$ is generic, then the limit
\begin{align}
\alpha_{f}(x) = \lim_{n \to \infty} \max\{1, h_{H}(f^{n}(x))\}^{1/n}
\end{align}
exists. 
\end{theorem}
The key of the proof is certain recursive inequality of pull-backs of divisors 
proven in \cite{xie2024algebraic}.
See \cref{thm:existence_ad_generic,lem:htgrowth} for stronger statements.

Dynamical Mordell-Lang conjecture implies that a Zariski dense orbit is automatically generic.  
(See \cite{DMLbook,xie2023around} for dynamical Mordell-Lang conjecture. It is originally proposed for morphisms on quasi-projective varieties,
but generalized and studied for rational maps as well.)
If this is the case for all dominant rational maps, the above theorem implies the existence of arithmetic degree
for all points. Indeed if the original orbit is not Zariski dense, just replace $X$ with an irreducible component 
of the Zariski closure of the orbit.

Let us say a triple $(X, f, x)$, consisting of a variety $X$ over $\QQ$, a dominant rational map $f \colon X \dashrightarrow X$,
and a point $x \in X_{f}(\QQ)$, satisfies the dynamical Mordell-Lang property if
for any Zariski closed subset $Z \subset X$, the return set $\{ n \in \Z_{\geq 0} \mid f^{n}(x) \in Z\}$
is a finite union of arithmetic progressions.

\begin{corollary}\label{intro:existence-ad-dml}
Let $f \colon X \dashrightarrow X$ be a dominant rational map on a projective variety $X$ defined over $\QQ$.
Let $H$ be an ample divisor on $X$ and fix Weil height function $h_{H}$ associated with $H$.
Let $x \in X_{f}(\QQ)$.
If $(X, f, x)$ satisfies the dynamical Mordell-Lang property, then
\begin{align}
\alpha_{f}(x) = \lim_{n \to \infty} \max\{1, h_{H}(f^{n}(x))\}^{1/n}
\end{align}
exists. 
\end{corollary}

As Dynamical Mordell-Lang conjecture is known for \'etale morphisms \cite{BGT10DMLetale}
and self-morphisms on $\A^{2}_{\QQ}$ \cite{Xi17}, we have the followings.

\begin{corollary}\label{intro:existence-ad-auto}
Let $X$ be a quasi-projective variety over $\QQ$. 
Let $X \hookrightarrow  P$ be any immersion to a projective variety $P$
over $\QQ$ and let $H$ be an ample divisor on $P$.
Fix a Weil height function $h_{H}$ associated with $H$.
Let $f \colon X \longrightarrow X$ be an \'etale morphism.
Then for any point $x \in X(\QQ)$, the limit
\begin{align}
\alpha_{f}(x) = \lim_{n \to  \infty} \max\{1, h_{H}(f^{n}(x))\}^{1/n}
\end{align}
exists.
Here $h_{H}$ is regarded as a function on $X(\QQ)$ via the embedding $X \hookrightarrow  P$.
\end{corollary}

\begin{remark}
We refer \cite[Remark 2.4]{matsuzawa2023recent} and \cite[Lemma 3.8]{JSXZ21} for the definition of arithmetic degree on quasi-projective varieties,
namely the independence of the choice of ambient projective varieties.
\end{remark}

\begin{corollary}
Let $f \colon \A^{2}_{\QQ} \longrightarrow \A^{2}_{\QQ}$ be a dominant morphism.
Let $h$ be the naive height function on $\P^{2}(\QQ)$ and restrict it to $\A^{2}(\QQ)$
via the open immersion $\A^{2} \longrightarrow \P^{2}, (a,b) \mapsto (a:b:1)$.
Then for any $x \in \A^{2}(\QQ)$, the limit
\begin{align}
\alpha_{f}(x) = \lim_{n \to  \infty} \max\{1, h(f^{n}(x))\}^{1/n}
\end{align}
exists.
\end{corollary}

The proof of \cref{intro:thm-ad-exists}, namely \cref{lem:htgrowth}
has an application to dynamical Lang-Siegel problem.
Dynamical Lang-Siegel problem in a classical form asks if
the size of coordinates of orbits are asymptotically the same.
This problem can be reformulated by using local height function in the following form:
for a dominant rational map $f \colon X \dashrightarrow X$ defined over a number field $K$,
a proper closed subscheme $Y \subset X$, and a point $x \in X(K)$, when do we have
\begin{align}
\lim_{n \to \infty} \frac{\sum_{v \in S}\l_{Y,v}(f^{n}(x))}{h_{H}(f^{n}(x))} = 0 \ ?
\end{align}
Here $S$ is a finite set of places of $K$, $\l_{Y,v}$ is a local height function associated with $Y$,
and $h_{H}$ is any height function associated with any ample divisor.
See \cite{Sil93dml,mat2023growthlocal,Yas2015integral,Yas2014deviation} for some of the results related to this problem.

To state our results on dynamical Lang-Siegel problem,
let us introduce some notation.
Let $K$ be a number field.
The standard proper set of absolute values on $K$ is denoted by $M_{K}$ (see Convention below).
Let $X$ be a projective geometrically integral variety over $K$.
For a dominant rational map $f \colon X \dashrightarrow X$ defined over $K$, we write 
\begin{align}
X_{f}(K) = \left\{ x \in X(K) \ \middle|\ f^{n}(x) \notin I_{f},\ n \geq 0 \right\}.
\end{align}
This is a set of points for which the forward $f$-orbits are well-defined.
Let us introduce a backward version $X_{f}^{\rm back}$, that is 
$X_{f}^{\rm back}$ is, roughly speaking,  the largest subset over which $f^{n}$ are finite for all $n \geq 0$.
See \cref{def:Xfback} and \cref{section:uscbackconvergence} for the formal definition.
For $x \in X$ at which $f$ is quasi-finite, let
\[
e_{f}(x) = l_{ \mathcal{O}_{X,x}}(\mathcal{O}_{X,x}/f^{*} \mathfrak{m}_{f(x)}\O_{X,x})
\]
where $l_{ \mathcal{O}_{X,x}}$ stands for the length as an $\mathcal{O}_{X,x}$-module, 
$ \mathfrak{m}_{f(x)}$ is the maximal ideal of the local ring $\O_{X,f(x)}$, 
and $f^{*} \colon \O_{X, f(x)} \longrightarrow \O_{X,x}$ is the induced local homomorphism.
For a finite subset $Y \subset X_{f}^{\rm back}$, we define
\begin{align}\label{intro:def:ebar}
e(f; Y) := \lim_{n \to \infty} \big( \max\{ e_{f^{n}}(z) \mid z \in f^{-n}(Y)  \}  \big)^{1/n}.
\end{align}
See \cref{sec:mult} for more about this definition.

The following theorem is a rational map version of the author's results \cite[Theorem 1.11]{mat2023growthlocal},
which is a generalization of Silverman's result \cite[Theorem E]{Sil93dml}.
Note that \cite[Theorem 1.11]{mat2023growthlocal} does not assume the orbit is generic,
while the following theorem does.

\begin{theorem}\label{intro:thm-dls-rat-map}
Let $f \colon X \dashrightarrow X$ be a dominant rational map on a geometrically integral projective variety $X$ 
defined over a number field $K$.
Let $H$ be an ample Cartier divisor on $X$ and $h_{H}$ an associated Weil height function.
Let $Y \subset X$ be a $0$-dimensional closed subscheme and fix local height function $\{ \l_{Y,v} \}_{v \in M_{K}}$ associated with $Y$.
Suppose $Y \subset X_{f}^{\rm back}$.
Let $x \in X_{f}(K)$ be such that $O_{f}(x)$ is generic.
If 
\begin{align}
e(f;Y) < \alpha_{f}(x),
\end{align}
then for any finite subset $S \subset M_{K}$, we have
\begin{align}
\lim_{n \to \infty} \frac{\sum_{v \in S}\l_{Y,v}(f^{n}(x))}{h_{H}(f^{n}(x))} = 0.
\end{align}
\end{theorem}

In \cite{mat2023growthlocal}, the author proves that the similar result still holds for positive dimensional $Y$
if $X$ is smooth, $f$ is a morphism (i.e.\ $I_{f}=  \emptyset$), and one assumes Vojta's conjecture.
It seems difficult to avoid the conjecture at this moment, 
but we are able to prove a weaker statement without assuming it.
Note that when $f$ is a surjective self-morphism, it is finite and hence $X_{f}^{\rm back} = X$.
Moreover, the limit \cref{intro:def:ebar} exists for any closed subset $Y \subset X$
(cf.\ \cite[section 4]{mat2023growthlocal}). 

\begin{theorem}\label{intro:thm:grwothlocalhtBD}
Let $f \colon X \longrightarrow X$ be a surjective morphism on a smooth projective geometrically irreducible variety $X$ 
defined over a number field $K$.
Let $H$ be an ample Cartier divisor on $X$ and $h_{H}$ an associated Weil height function.
Let $Y \subset X$ be a proper closed subscheme and fix a local height function $\{ \l_{Y,v} \}_{v \in M_{K}}$ associated with $Y$.
Let $x \in X(K)$ be a point.
Let $S \subset M_{K}$ be a finite set of places.
Suppose
\begin{enumerate}
\item $Y$ contains no positive dimensional $f$-periodic subvariety;
\item $e(f;Y) < \alpha_{f}(x)$.
\end{enumerate}
Then if an infinite sequence $\{n_{j}\}_{j=1}^{\infty}$ satisfies
\begin{align}
\liminf_{j\to \infty} \frac{\sum_{v \in S}\l_{Y,v}(f^{n_{j}}(x))}{h_{H}(f^{n_{j}}(x))} > 0,
\end{align}
then it has Banach density zero (cf.\ \cref{def:banach-density}).
\end{theorem}

\noindent
{\bf convention}
In this paper, we work over a number field or a field of characteristic zero.
\begin{itemize}
\item An  \emph{algebraic scheme} over a field $k$ is a separated scheme of finite type over $k$;
\item A  \emph{variety} over $k$ is an algebraic scheme over $k$ which is irreducible and reduced;
\item A  \emph{nice variety} over a field $k$ is a smooth projective geometrically irreducible scheme over $k$;
\item Let $X$ be a scheme over a field $k$ and $k \subset k'$ be a field extension. 
 \emph{The base change} $X \times_{\Spec k}\Spec k'$ is denoted by $X_{k'}$.
For an ``object" $A$ on $X$, we sometimes use the notation $A_{k'}$ to express the base change of $A$ to $k'$ 
without mentioning to the definition of the base change if the meaning is clear.
\item For a self-morphism $f \colon X \longrightarrow X$ of an algebraic scheme over $k$ and a 
point $x$ of $X$ (scheme point or $k'$-valued point where $k'$ is a field containing $k$), the  \emph{$f$-orbit of $x$}
is denoted by $O_{f}(x)$, i.e. $O_{f}(x) = \{ f^{n}(x) \mid n=0,1,2, \dots\}$.
The same notation is used for dominant rational map $f \colon X \dashrightarrow X$ on a variety $X$ defined over $k$
and $x \in X_{f}(k) = \{ x \in X(k) \mid f^{n}(x) \notin I_{f}, \ n \geq 0\}$.
\item
For a number field $K$, let $M_{K}$ denote the set of absolute values normalized as in \cite[p11 (1.6)]{BG06}.
Namely, Namely, if $K=\Q$, then $M_{\Q}=\{|\ |_{p} \mid \text{$p=\infty$ or a prime number} \}$ with
\begin{align}
 &|a|_{\infty} =
 \begin{cases}
 a \quad \text{if $a\geq0$}\\
 -a \quad \text{if $a<0$}
 \end{cases}
 \\
 & |a|_{p} = p^{-n} \quad \txt{if $p$ is a prime and  $a=p^{n}\frac{k}{l}$ where\\ $k,l$ are non zero integers coprime to $p$.}
\end{align}
For a number field $K$, $M_{K}$ consists of the following absolute values:
\begin{align}
|a|_{v} = |N_{K_{v}/\Q_{p}}(a)|_{p}^{1/[K:\Q]} 
\end{align}
where $v$ is a place of $K$ which restricts to $p = \infty$ or a prime number.

\item
Let $X$ be a projective variety over a number field $K$.
Let $Y \subset X$ be a closed subscheme.
A local height function associated with $Y$ is denoted by $\l_{Y,v}$ or $\{ \l_{Y,v}\}_{v \in M_{K}}$.
When $v \in M_{K}$ is extended to an absolute value on the algebraic closure $ \overline{K}$,
we can extend $\l_{Y,v}$ to a function on $(X \setminus Y)( \overline{K})$. This extended local height function
is also denoted by $\l_{Y,v}$.
(cf.  \cite[section 2]{mat2023growthlocal}, \cite{BG06}, \cite{HS00}, \cite{La83}.)
\item
The length of a module $M$ over a ring $A$ is denoted by $l_{A}(M)$.
\end{itemize}

\begin{acknowledgements}
The author would like to thank Kaoru Sano for helpful discussions.
He would also like to thank Joseph Silverman for his comments on the first version of this paper.
The author would like to thank the referee for valuable comments, especially on \cref{thm:existence_ad_generic} and its proof.
The author is supported by JSPS KAKENHI Grant Number JP22K13903.
\end{acknowledgements}

\section{Height growth and arithmetic degree}

In this section, we prove \cref{intro:thm-ad-exists,intro:existence-ad-auto}.

We use the following quantity, introduced and called cohomological Lyapunov exponent in \cite{xie2024algebraic}.
\begin{definition}
Let $f \colon X \dashrightarrow X$ be a dominant rational map on a projective variety $X$ defined over $\QQ$.
The $i$-th dynamical degree of $f$ is denoted by $d_{i}(f)$ for $i = 0, \dots, \dim X$.
We set
\begin{align}
\mu_{i}(f) = \frac{d_{i}(f)}{d_{i-1}(f)}
\end{align}
for $i = 1, \dots, \dim X$ and set $\mu_{\dim X + 1}(f) = 0$.
Let us say $f$ has dynamical peak $p \in \{1, \dots, \dim X\}$ if
\begin{align}
\mu_{p}(f) > 1 \ \text{and}\ \mu_{p+1}(f) \leq 1.
\end{align}
Note that having dynamical peak is equivalent to the condition $d_{1}(f) > 1$.
\end{definition}

By log concavity of dynamical degrees, we have
\begin{align}
\mu_{1}(f) \geq  \mu_{2}(f) \geq  \cdots \geq \mu_{p}(f) > 1 \geq \mu_{p+1}(f).
\end{align}

The following \cref{thm:existence_ad_generic} is \cref{intro:thm-ad-exists} with some additional 
statement.
The key of the proof is \cref{lem:htgrowth}, which proves stronger 
statement on the growth rate of ample height function along orbits.

\begin{theorem}\label{thm:existence_ad_generic}
Let $f \colon X \dashrightarrow X$ be a dominant rational map on a projective variety $X$ defined over $\QQ$.
Let $H$ be an ample divisor on $X$ and fix Weil height function $h_{H}$ associated with $H$.
Let $x \in X_{f}(\QQ)$.
If $O_{f}(x)$ is generic, then
\begin{align}
\alpha_{f}(x) := \lim_{n \to \infty} h_{H}(f^{n}(x))^{1/n}
\end{align}
exists. Moreover, $ \alpha_{f}(x)$ takes one of the following values:
\begin{align}
\mu_{1}(f) = d_{1}(f), \mu_{2}(f), \dots , \mu_{\dim X}(f).
\end{align}
Note that $ \alpha_{f}(x) \geq 1$ by definition, so we can ignore $\mu_{i}(f)$'s that are strictly less than $1$.
\end{theorem}

\begin{remark}
Let us note that under the assumption of the theorem,
$ \alpha_{f}(x)$ is expected to be equal to $ d_{1}(f)$ according to Kawaguchi-Silverman conjecture.
\end{remark}

For $x \in X_{f}(\QQ)$, we write 
\begin{align}
&\overline{\alpha}_{f}(x) = \limsup_{n \to \infty} \max\{1, h_{H}(f^{n}(x)) \}^{1/n}\\
&\underline{\alpha}_{f}(x) = \liminf_{n \to \infty} \max\{1, h_{H}(f^{n}(x)) \}^{1/n}.
\end{align}
We obviously have $ 1 \leq \underline{\alpha}_{f}(x) \leq \overline{\alpha}_{f}(x)$.
Moreover, we always have $ \overline{\alpha}_{f}(x) \leq d_{1}(f) = \mu_{1}(f)$ \cite[Theorem 1.4]{Ma20a},\cite[Proposition 3.11]{JSXZ21}.

\begin{lemma}\label{lem:htgrowth}
Let $f \colon X \dashrightarrow X$ be a dominant rational map on a projective variety $X$ defined over $\QQ$.
Let $x \in X_{f}(\QQ)$ be such that $O_{f}(x)$ is generic.
Let $H$ be an ample Cartier divisor on $X$ and 
let $h_{H}$ be an associated Weil height function such that $h_{H} \geq 1$.
Let $l \in \{1, \dots, \dim X\}$ be such that $\mu_{l}(f) > \mu_{l+1}(f)$.
If 
\begin{align}
\overline{\alpha}_{f}(x)= \limsup_{n \to \infty} h_{H}(f^{n}(x))^{1/n} > \mu_{l+1}(f),
\end{align}
then
there is $C \in \R_{>0}$ such that for any $\eta \in (0, 1)$,
there is $m_{0} \in \Z_{\geq 1}$ with the following property.
For any given $m \in \Z_{\geq m_{0}}$, we have
\begin{align}
h_{H}((f^{m})^{n+k}(f^{s}(x)) ) \geq C (\eta \mu_{l}(f))^{mk} h_{H}((f^{m})^{n}(f^{s}(x))).
\end{align}
for all large enough $s, n \in \Z_{\geq 0}$ and for all $k \in \Z_{\geq 0}$.
The logical order of the quantifiers are as follows:
\begin{align}
&\exists C \in \R_{>0}, \forall \eta \in (0, 1), \exists m_{0} \in \Z_{\geq 1}, \forall m \in \Z_{\geq m_{0}}, \\
&\exists s_{0} \in \Z_{\geq0}, \exists n_{0} \in \Z_{\geq 0}, \forall s \in \Z_{\geq s_{0}}, \forall n \in \Z_{\geq n_{0}}, \forall k \in \Z_{\geq 0}.
\end{align}
\end{lemma}

\begin{proof}
We write $\mu_{i}(f) = \mu_{i}$.
By \cite[Theorem 3.6]{xie2024algebraic}, for any $\e \in (0,1)$, there is $m_{\e} \in \Z_{\geq 1}$ such that
\begin{align}
(f^{2m})^{*}H + \mu_{l}^{m}\mu_{l+1}^{m} H - \e^{m} \mu_{l}^{m} (f^{m})^{*}H
\end{align}
is big as an element of $ \widetilde{\Pic}(X)_{\R}$ for all $m \geq m_{\e}$.
Here $ \widetilde{\Pic}(X)_{\R}$ is the colimit of $\Pic (X')_{\R}$ where $X'$ runs over birational models of $X$.
See \cite{xie2024algebraic} for the detail.

Let us fix an arbitrary $\e_{1} \in (0,1)$ with the following properties:
\begin{align}\label{choice-of-e1}
\frac{\mu_{l+1}}{\e_{1}^{3} \mu_{l}} < 1\quad \text{and}\quad \overline{\alpha}_{f}(x) > \e_{1}^{-2} \mu_{l+1}.
\end{align}
Let us fix $e \in (\e_{1},1)$ and set
\begin{align}
C = 1 - \frac{\e_{1}}{e}.
\end{align}
Let us take arbitrary $\eta \in (0,1)$.
Then take $\e_{2} \in (e, 1)$ such that $\e_{2}^{2} \geq \eta$. Note that we still have
\begin{align}
\frac{\mu_{l+1}}{\e_{2}^{3} \mu_{l}} < 1.
\end{align}
Then there is $m_{0} \in \Z_{\geq 1}$ with the following properties:
\begin{align}
m_{0} \geq \max\{ m_{\e_{1}}, m_{\e_{2}} \}
\end{align}
and for all $m \geq m_{0}$ we have
\begin{align}
\frac{1}{2} (\e_{i} \mu_{l})^{m} > (\e_{i}^{2}\mu_{l})^{m} + (\e_{i}^{-2}\mu_{l+1})^{m} \label{mpineq}
\end{align}
for $i = 1,2$.

Fix an $m \geq m_{0}$.
Fix a birational model $\pi \colon X_{\pi} \longrightarrow X$ such that 
\begin{align}
&\text{
$\f := f^{m} \circ \pi$ and $\psi := f^{2m} \circ \pi$ are morphisms,
}\\
&
\text{
$\psi^{*}H + \mu_{l}^{m} \mu_{l+1}^{m} \pi^{*}H - \e_{i}^{m} \mu_{l}^{m}\f^{*}H$ is big as an element of $\Pic(X_{\pi})_{\R}$ for $i=1,2$.
}
\end{align}
See the following diagram:
\begin{equation}
\begin{tikzcd}
X_{\pi} \arrow[d, "\pi", swap] \arrow[rd, "\f"] \arrow[rrd, "\psi", bend left=20]& & \\
X \arrow[r, dashed, "f^{m}", swap] & X \arrow[r,dashed, "f^{m}", swap] & X
\end{tikzcd}.
\end{equation}
Thus there is a proper closed subset $Z \subset X_{\pi}$ and $C_{1} \in \R$ such that
\begin{align}
h_{H} \circ \psi + \mu_{l}^{m} \mu_{l+1}^{m} h_{H} \circ \pi - \e_{i}^{m} \mu_{l}^{m} h_{H} \circ \f \geq C_{1} \label{htineq}
\end{align}
for $i=1,2$ on $(X_{\pi} \setminus Z)(\QQ)$.

Since $O_{f}(x)$ is generic, there is $s_{0}\geq 0$ such that
\begin{align}
f^{s}(x) \not\in I_{\pi^{-1}} \cup \pi(Z)
\end{align}
for all $s \geq s_{0}$.
We fix an $s \geq s_{0}$.

Set $x' = f^{s}(x)$ and $g = f^{m}$. Then note that we have 
\begin{align}
&x' \in X_{g}(\QQ);\\
& g^{n}(x') \not\in I_{\pi^{-1}} \cup \pi(Z) \ \text{for all $n\geq 0$}.
\end{align}
By \cref{htineq}, we have
\begin{align}
h_{H}(g^{n+2}(x')) + \mu_{l}^{m} \mu_{l+1}^{m} h_{H}(g^{n}(x')) - \e_{i}^{m}\mu_{l}^{m}h_{H}(g^{n+1}(x')) \geq C_{1}
\end{align}
for $i = 1,2$ and for all $n\geq 0$.
Since $O_{g}(x')$ is infinite and all the points of  $O_{g}(x')$ are defined over a number field,
we have $h_{H}(g^{n}(x')) \to \infty$ as $n \to \infty$.
Thus there is $n_{1} \in \Z_{\geq 0}$ such that
\begin{align}
C_{1} + \frac{1}{2}\e_{i}^{m}\mu_{l}^{m}h_{H}(g^{n+1}(x')) \geq 0
\end{align}
for $i=1,2$ and for all $n \geq n_{1}$.
Thus for $n \geq n_{1}$, we have
\begin{align}
h_{H}(g^{n+2}(x')) + \mu_{l}^{m} \mu_{l+1}^{m} h_{H}(g^{n}(x')) - \frac{1}{2} \e_{i}^{m}\mu_{l}^{m}h_{H}(g^{n+1}(x')) \geq 0
\end{align}
or equivalently,
\begin{align}
h_{H}(g^{n+2}(x')) + (\e_{i}^{2}\mu_{l})^{m} (\e_{i}^{-2}\mu_{l+1})^{m} h_{H}(g^{n}(x')) - \frac{1}{2} \e_{i}^{m}\mu_{l}^{m}h_{H}(g^{n+1}(x')) \geq 0
\end{align}
for $i=1,2$.
By \cref{mpineq}, we get
\begin{align}
h_{H}(g^{n+2}(x')) &+ (\e_{i}^{2}\mu_{l})^{m} (\e_{i}^{-2}\mu_{l+1})^{m} h_{H}(g^{n}(x')) \\
&-  ((\e_{i}^{2}\mu_{l})^{m}+ (\e_{i}^{-2}\mu_{l+1})^{m}) h_{H}(g^{n+1}(x')) \geq 0
\end{align}
and hence
\begin{equation}
\begin{aligned}
&h_{H}(g^{n+2}(x'))- (\e_{i}^{-2}\mu_{l+1})^{m} h_{H}(g^{n+1}(x')) \\
&\geq 
(\e_{i}^{2}\mu_{l})^{m} \big\{ h_{H}(g^{n+1}(x'))  - (\e_{i}^{-2}\mu_{l+1})^{m} h_{H}(g^{n}(x')) \big\} \label{ht-rec-ineq}
\end{aligned}
\end{equation}
for $i=1,2$ and for all $n \geq n_{1}$.
When $\mu_{l+1} = 0$, we get
\begin{align}
h_{H}(g^{n+2}(x'))  \geq 
(\e_{i}^{2}\mu_{l})^{m}  h_{H}(g^{n+1}(x'))   \label{ht-rec-ineq-mp1zero}
\end{align}
for $i=1,2$ and all $n \geq n_{1}$.

Suppose $\mu_{l+1} > 0$.
Then there is $n_{2} \geq n_{1}$ such that 
$h_{H}(g^{n_{2} + 1}(x')) > (\e_{1}^{-2}\mu_{l+1})^{m} h_{H}(g^{n_{2}}(x'))$.
Indeed, if 
\begin{align}
h_{H}(g^{n + 1}(x')) \leq  (\e_{1}^{-2}\mu_{l+1})^{m} h_{H}(g^{n}(x'))
\end{align}
for all $n \geq n_{1}$, then we have
\begin{align}
\overline{\alpha}_{g}(x') = \limsup_{n \to \infty} h_{H}(g^{n}(x'))^{1/n} \leq (\e_{1}^{-2}\mu_{l+1})^{m}.
\end{align}
But $ \overline{\alpha}_{g}(x') = \overline{\alpha}_{f}(x')^{m} = \overline{\alpha}_{f}(x)^{m}$ (cf.\ \cite[Lemma 2.7]{MW22}),
this contradicts to \cref{choice-of-e1}.

By \cref{ht-rec-ineq}, we have
\begin{align}
h_{H}(g^{n+1}(x'))  - (\e_{1}^{-2}\mu_{l+1})^{m} h_{H}(g^{n}(x')) > 0
\end{align}
for all $n \geq n_{2} $.
Thus by \cref{ht-rec-ineq}, for any $k \geq 0$ and $n \geq n_{2} + 1$, we have
\begin{align}
h_{H}(g^{n+k}(x')) 
&\geq 
h_{H}(g^{n+k}(x')) - (\e_{2}^{-2}\mu_{l+1})^{m} h_{H}(g^{n+k-1}(x')) \\
&\geq 
(\e_{2}^{2}\mu_{l})^{mk} \big\{ h_{H}(g^{n}(x')) - (\e_{2}^{-2}\mu_{l+1})^{m} h_{H}(g^{n-1}(x'))  \big\}\\
& \geq
(\e_{2}^{2}\mu_{l})^{mk} \bigg( 1 - \frac{\e_{1}^{2m}}{\e_{2}^{2m}} \bigg)h_{H}(g^{n}(x'))\\ 
& \geq
(\e_{2}^{2}\mu_{l})^{mk} \bigg( 1 - \frac{\e_{1}^{2m}}{e^{2m}} \bigg)h_{H}(g^{n}(x'))\\
& \geq
C(\e_{2}^{2}\mu_{l})^{mk} h_{H}(g^{n}(x')). \label{ht-rec-ineq-general}
\end{align}

Combining with \cref{ht-rec-ineq-mp1zero}, in either cases $\mu_{l+1}=0$ or not, we have
\begin{align}
h_{H}(g^{n+k}(x')) \geq C(\e_{2}^{2}\mu_{l})^{mk} h_{H}(g^{n}(x'))
\geq C(\eta\mu_{l})^{mk} h_{H}(g^{n}(x'))
\end{align}
for $k \geq 0$ and $n \geq n_{2} + 1$.
Thus setting $n_{0} = n_{2} + 1$, we are done.
\end{proof}

\begin{proof}[Proof of \cref{thm:existence_ad_generic}]
We may take $h_{H} \geq 1$.
Let us write $\mu_{i} = \mu_{i}(f)$.

We first recall that $1 \leq \overline{\alpha}_{f}(x) \leq d_{1}(f) = \mu_{1}$. 
Let $l \in \{1, \dots, \dim X\}$ be the smallest one with the property $\mu_{l+1} < \overline{\alpha}_{f}(x)$.
Then we have $ \mu_{l+1} <  \overline{\alpha}_{f}(x) \leq \mu_{l}$.
By \cref{lem:htgrowth}, 
there is $C >0$ with the following property: for $\eta \in (0,1)$ arbitrary close to $1$, there are $m \geq 1$, $n \geq 0$, $s \geq 0$
such that
\begin{align}
h_{H}((f^{m})^{n+k}(f^{s}(x))) \geq C (\eta \mu_{l})^{mk}h_{H}((f^{m})^{n}(f^{s}(x)))
\end{align}
for any $k \geq 0$.
Thus 
\begin{align}
\underline{\alpha}_{f^{m}}(f^{mn+s}(x)) = \liminf_{k \to \infty} h_{H}((f^{m})^{n+k}(f^{s}(x)))^{1/k} \geq (\eta \mu_{l})^{m}.
\end{align}
Hence
\begin{align}
\underline{\alpha}_{f}(x) = \underline{\alpha}_{f}(f^{mn +s}(x)) = \underline{\alpha}_{f^{m}}(f^{mn +s}(x))^{1/m} \geq \eta \mu_{l}
\end{align}
(cf.\ \cite[Lemma 2.7]{MW22} for the first and second equality).
By letting $\eta \to 1$, we get $ \underline{\alpha}_{f}(x) \geq \mu_{l}$.
Therefore we get $ \mu_{l} \leq \underline{\alpha}_{f}(x) \leq \overline{\alpha}_{f}(x) \leq \mu_{l}$ and we are done.
\end{proof}

\begin{remark}
In our proof of \cref{thm:existence_ad_generic}, genericness is used for avoiding indeterminacy locus
of the rational map to a birational model and the locus where the height of our big divisors are negative.
If our self-map is a morphism and the big divisors involved are all ample, then these issues disappear 
and the same argument works for infinite orbits.
This is the case, for example, when $X = (\P^{1}_{\QQ})^{d}$ and 
$f = g_{1} \times \cdots \times g_{d}$, where $g_{i} \colon \P^{1}_{\QQ} \longrightarrow \P^{1}_{\QQ}$
are surjective morphisms with $\deg g_{i} = a_{i} > 1$.
Indeed, every big divisor is ample on $X$, and so the ampleness automatically follows.
Moreover this example makes clear the limitation of our method.
If $a_{1} \geq a_{2} \geq \cdots \geq a_{d}$, then 
\begin{align}
&d_{1}(f) = a_{1}, d_{2}(f) = a_{1}a_{2}, \dots, d_{d}(f) = a_{1} \cdots a_{d} \quad \text{and}\\
& \mu_{1}(f) = a_{1}, \mu_{2}(f) = a_{2}, \dots , \mu_{d}(f) = a_{d}.
\end{align}
Each $a_{i}$ is realized as the arithmetic degree of $(x_{1}, \dots , x_{d}) \in X(\QQ)$
where $x_{j}$ are preperiodic points of $g_{j}$ for $ j \neq i$ and 
$x_{i}$ is a point with infinite $g_{i}$-orbit.

\end{remark}

\begin{proof}[Proof of \cref{intro:existence-ad-dml}]
If the orbit $O_{f}(x)$ is finite, $ \alpha_{f}(x) = 1$ and we are done.
Suppose $ O_{f}(x)$ is infinite.
Let $ Y = \overline{O_{f}(x)}$  be the Zariski closure of the orbit.
Take an irreducible component  $Z \subset Y$ such that $\dim Z = \dim Y > 0$
with the following property: there is an $r \in \Z_{\geq 1}$ such that $f^{r}$ induces a dominant rational map
$f^{r}|_{Z} \colon Z \dashrightarrow Z$, and $f^{s}(x) \in Z$ for some $s \in \Z_{\geq 0}$.

Then it is easy to check that $O_{f^{r}}(f^{s}(x))$ is Zariski dense in $Z$.
By our assumption,  $(f^{r}|_{Z}, Z, f^{s}(x))$ satisfies 
dynamical Mordell-Lang property, and
hence $O_{f^{r}}(f^{s}(x))$ is generic in $Z$.
Thus by \cref{thm:existence_ad_generic}, we have
\begin{align}
\underline{\alpha}_{f^{r}|_{Z}}(f^{s}(x)) = \overline{\alpha}_{f^{r}|_{Z}}(f^{s}(x)).
\end{align}
Thus by \cite[Lemma 2.7, Lemma 2.8]{MW22},
\begin{align}
\underline{\alpha}_{f}(x) &= \underline{\alpha}_{f}(f^{s}(x)) = \underline{\alpha}_{f^{r}}(f^{s}(x))^{1/r}
= \underline{\alpha}_{f^{r}|_{Z}}(f^{s}(x))^{1/r} \\
&= \overline{\alpha}_{f^{r}|_{Z}}(f^{s}(x))^{1/r} = \overline{\alpha}_{f^{r}}(f^{s}(x))^{1/r} = \overline{\alpha}_{f}(f^{s}(x)) = \overline{\alpha}_{f}(x)
\end{align}
and we are done.
\end{proof}

\begin{proof}[Proof of \cref{intro:existence-ad-auto}]
This follows from \cref{intro:existence-ad-dml} since 
dynamical Mordell-Lang conjecture holds for \'etale morphisms \cite{BGT10DMLetale}.
\end{proof}

\section{Dynamical Lang-Siegel problem for rational map}\label{sec:dml-rat-map}

In this section, we prove \cref{intro:thm-dls-rat-map}.

\begin{definition}\label{def:Xfback}
Let $X$ be a geometrically integral projective variety over a number field $K$.
Let $f \colon X \dashrightarrow X$ be a dominant rational map.
We set
\begin{equation}
\begin{tikzcd}
\G_{f} \arrow[d,"p_{1}",swap] \arrow[dr, "p_{2}"]& \\
X \arrow[r, "f", swap,dashed]& X
\end{tikzcd}
\end{equation}
where $\G_{f}$ is the graph of $f$.
Let $U \subset X$ be the largest open subset such that
\begin{align}
&p_{2} \colon p_{2}^{-1}(U) \longrightarrow U \ \text{is finite and}\\
&p_{2}^{-1}(U) \cap p_{1}^{-1}( I_{f}) =  \emptyset.
\end{align}
We define a sequence of open subsets $U_{n} \subset X$, $n=0,1,2, \dots$ inductively as follows.
We set $U_{0} = X$.
Let $n \geq 0$ and suppose we have constructed $U_{n}$  so that
$f^{n}$ is finite over $U_{n}$ (i.e. there is an open subset $V \subset X \setminus I_{f^{n}}$ such that $f^{n}|_{V} \colon V \longrightarrow U_{n}$ is finite).
Then define 
\begin{align}
U_{n+1} = U_{n} \setminus f^{n}( f^{-n}(U_{n}) \setminus U ).
\end{align}
By construction, we have
\begin{equation}
\begin{tikzcd}
X \arrow[r, "f", dashed] & X \arrow[r, "f", dashed]  & \cdots \arrow[r, "f", dashed]  & X \arrow[r, "f", dashed]  & X \\
f^{-n}(U_{n}) \arrow[u, phantom, "\subset", sloped] \arrow[r, "f"] & f^{-(n-1)}(U_{n}) \arrow[u, phantom, "\subset", sloped]  \arrow[r, "f"]  & \cdots  \arrow[r, "f"] & f^{-1}(U_{n}) \arrow[u, phantom, "\subset", sloped] \arrow[r, "f"] & U_{n}  \arrow[u, phantom, "\subset", sloped] \\
& U \arrow[u,phantom, "\supset", sloped]  && U \arrow[u,phantom, "\supset", sloped]  & U \arrow[u,phantom, "\supset", sloped]  
\end{tikzcd}
\end{equation}
with all the morphisms in the second row are finite.
Then we define
\begin{align}
X_{f}^{\rm back} := \bigcap_{n \geq 0} U_{n}.
\end{align}
This is a Zariski topological space with the induced topology from $X$.
See \cref{subsec:finlociratmap} for more about this construction and basic properties.
\end{definition}
Note that for a subset $Y \subset X_{f}^{\rm back}$, we can define its inverse image $f^{-k}(Y) \subset X$ naturally
for $k \geq 0$. Moreover, we have $f^{-(k+l)}(Y) = f^{-l}(f^{-k}(Y))$ for all $k,l \geq 0$.
By \cref{lem:maximalityUn}, we can easily see $X_{f}^{\rm back} \subset X_{f^{m}}^{\rm back}$ for all $m \geq 1$.

For a field extension $K \subset L$, we have
\begin{align}
(X_{L} \longrightarrow X)^{-1}(X_{f}^{\rm back}) \subset (X_{L})_{f_{L}}^{\rm back}.
\end{align}

\begin{definition}\label{def:ebar}
Let $X$ be a geometrically integral projective variety over a number field $K$.
Let $f \colon X \dashrightarrow X$ be a dominant rational map.
Let $Y \subset X$ be a closed subscheme.
Suppose $\dim Y = 0$ and $Y \subset X_{f}^{\rm back}$.
We define
\begin{align}
e(f; Y) := \lim_{n \to \infty} \bigg( \max\{ e_{f^{n}}(z) \mid z \in f^{-n}(Y)  \}  \bigg)^{1/n}.
\end{align}
See \cref{sec:mult} for the existence of this limit.
\end{definition}
We refer \cref{sec:mult} for the basic properties of this quantity.
Some important properties are
\begin{itemize}
\item $e(f^{m} ; Y) = e(f;Y)^{m}$ for $m \geq 1$;
\item $e(f ; Y) = e(f_{ \overline{K}} ; Y_{ \overline{K}})$.
\end{itemize}
Moreover, the limit
\begin{align}
e_{f, -}(x) := \lim_{n \to \infty} \bigg( \max\{ e_{f^{n}}(z) \mid z \in f^{-n}(x)  \}  \bigg)^{1/n}
\end{align}
exists for $x \in X_{f}^{\rm back}$ and $e_{f,-}$ is upper semicontinuous on $X_{f}^{\rm back}$.
We can easily see that $e_{f,-}=1$ at the generic point of $X_{f}^{\rm back}$.
Since $e(f; Y) = \max_{y \in Y} e_{f,-}(y)$, there is an open dense subset $X_{f}^{\rm back}$
such that $e(f;Y)=1$ as soon as $Y$ is contained in it.

\begin{theorem}[\cref{intro:thm-dls-rat-map}]\label{thm:dls-coh-hyp}
Let $f \colon X \dashrightarrow X$ be a dominant rational map on a geometrically integral projective variety $X$ defined over a number field $K$.
Let $H$ be an ample Cartier divisor on $X$ and 
let $h_{H}$ be an associated Weil height function.
Let $Y \subset X$ be a $0$-dimensional closed subscheme such that $Y \subset X_{f}^{\rm back}$.
Let $x \in X_{f}(K)$ be such that $O_{f}(x)$ is generic.
If 
\begin{align}
e(f;Y) < \alpha_{f}(x),
\end{align}
then for any finite subset $S \subset M_{K}$, we have
\begin{align}
\lim_{n \to \infty} \frac{\sum_{v \in S}\l_{Y,v}(f^{n}(x))}{h_{H}(f^{n}(x))} = 0.
\end{align}
\end{theorem}

\begin{proof}
To prove the theorem, we may assume $S$ is a singleton: $S =\{v\}$.
Taking base extension of $K$ by a finite extension (and extending $v$ to the finite extension and then normalizing appropriately),
we may assume that the residue field of every point of $Y$ is $K$.
Since $\l_{Y,v} \leq N \sum_{y \in Y}\l_{y,v}$ for some $N \geq 1$ and $ e(f;y) \leq e(f;Y)$ for $y \in Y$,
 we may assume that $Y$ is a single rational point $y$
(with reduced structure).
We may also assume that $H$ is very ample and $h_{H} \geq 1$.

By $e(f;Y) < \alpha_{f}(x)$, we in particular have $ d_{1}(f) \geq \alpha_{f}(x) > 1$.
Let $p$ be the dynamical peak of $f$.
By \cref{thm:existence_ad_generic} and $ \alpha_{f}(x) > 1$, there is $l \in \{ 1, \dots, p\}$
such that $ \alpha_{f}(x) = \mu_{l}(f)$.
In particular, we have $ \overline{\alpha}_{f}(x) = \alpha_{f}(x) > \mu_{l+1}(f)$.
By \cref{lem:htgrowth} and the assumption $ e(f;y) < \alpha_{f}(x) = \mu_{l}(f)$, there are 
\begin{align}
\eta \in (0,1), C >0, m \in \Z_{\geq 1}, s_{0} \in \Z_{\geq 0}, n_{0} \in \Z_{\geq 0}
\end{align}
such that if we set $g = f^{m}$, then we have
\begin{align}
&e(f;y) < \eta \mu_{l}(f) = \eta \alpha_{f}(x) ; \\
&h_{H}(g^{n+k}(f^{s}(x))) \geq C (\eta \alpha_{f}(x))^{mk} h_{H}(g^{n}(f^{s}(x))) \label{ineqht+kstep}
\end{align}
for all $s \geq s_{0}$, $n \geq n_{0}$, and $k \geq 0$.
Note that we have $e(g;y) = e(f^{m};y) = e(f;y)^{m} <( \eta \alpha_{f}(x))^{m}$.

Now let us fix an arbitrary $\e>0$.
Then there is $k \geq 1$ such that
\begin{align}
\max\{ e_{g^{k}}(z) \mid z \in g^{-k}(y)  \} < \e (\eta \alpha_{f}(x))^{mk}. \label{ineqevsmu}
\end{align}
We fix such $k$.
For arbitrary $s \geq s_{0}$, write $x' = f^{s}(x)$.
Since $y \in X_{f}^{\rm back}$, for any $n \geq n_{0} + k$, we have
\begin{align}
\l_{y,v}(g^{n}(x')) \leq \l_{g^{-k}(y), v}(g^{n-k}(x')) + C_{1}
\end{align}
where $g^{-k}(y)$ is equipped with the scheme structure as a fiber product.
Here $C_{1}$ is a constant depending on everything 
(in particular $k$ and $g^{-k}(y)$) we have chosen but is independent of $s$ and $n$. 

Note that we have
\begin{align}
\m_{z}^{e_{g^{k}}(z)} \subset \I_{g^{-k}(y) , z}
\end{align}
for each $z \in g^{-k}(y)$, where $\m_{z}$ is the maximal ideal of the local ring $\O_{X,z}$
and $\I_{g^{-k}(y)}$ is the ideal sheaf of $g^{-k}(y) \subset X$ (cf.\ \cite[Lemma 4.5]{mat2023growthlocal}).
By \cref{thm:roth}, which is essentially Roth's theorem, 
\begin{align}
\l_{g^{-k}(y) , v} \leq  3  \max\{ e_{g^{k}}(z) \mid z \in g^{-k}(y)  \} h_{H}
\end{align}
on $X(K)$ except for finitely many points.
Thus if $n$ is large enough so that $g^{n-k}(x')$ is not contained in the exceptional set, we have
\begin{align}
\l_{y,v}(g^{n}(x')) \leq 3 \max\{ e_{g^{k}}(z) \mid z \in g^{-k}(y)  \} h_{H}(g^{n-k}(x')) + C_{1}.
\end{align}
Thus by \cref{ineqht+kstep} and \cref{ineqevsmu}, we get
\begin{align}
\frac{\l_{y,v}(g^{n}(x'))}{h_{H}(g^{n}(x'))} \leq \frac{3 \e}{C} + \frac{C_{1}}{h_{H}(g^{n}(x'))}
\end{align}
and hence
\begin{align}
\limsup_{n \to \infty} \frac{\l_{y,v}(g^{n}(x'))}{h_{H}(g^{n}(x'))}  \leq \frac{3 \e}{C}.
\end{align}
Since $\e$ is arbitrary, we get
\begin{align}
\lim_{n \to \infty} \frac{\l_{y,v}(g^{n}(x'))}{h_{H}(g^{n}(x'))} 
= \lim_{n \to \infty} \frac{\l_{y,v}(f^{mn+s}(x))}{h_{H}(f^{mn+s}(x))}   = 0.
\end{align}
Since $s \geq s_{0}$ is arbitrary, we get
\begin{align}
\lim_{n \to \infty} \frac{\l_{y,v}(f^{n}(x))}{h_{H}(f^{n}(x))}   = 0.
\end{align}
\end{proof}

\begin{example}
Let $X$ be a geometrically integral projective variety over a number field $K$ and $U \subset X$ be a non-empty open subset.
Let $f \colon U \longrightarrow U$ be an \'etale finite morphism (e.g.\ automorphism).
The dominant rational self-map on $X$ induced by $f$ is also denoted by $f$.
Then we have $U \subset X_{f}^{\rm back}$.
For any closed subscheme $Y \subset X$ such that $Y \subset U$, we have
$e(f;Y) = 1$.
Let $x \in U(K)$ be such that the orbit $O_{f}(x)$ is Zariski dense.
Since dynamical Mordell-Lang conjecture is proven for \'etale morphisms \cite{BGT10DMLetale}, $O_{f}(x)$ is generic.
Thus if $\dim Y = 0$ and $ \alpha_{f}(x) > 1$, then 
we can apply \cref{thm:dls-coh-hyp} to this situation.
\end{example}

\section{Banach density and dynamical Lang-Siegel problem}


In this section, we prove \cref{intro:thm:grwothlocalhtBD}.

\begin{definition}\label{def:banach-density}
Let $A \subset \Z_{\geq 0}$.
The (upper) Banach density of $A$ is defined as
\[
\delta(A) = \limsup_{d \to \infty} \max_{n \in \Z_{\geq 0}} \frac{\sharp(A \cap [n, n+d])}{d+1}.
\]
We say $A$ has Banach density zero if $ \delta(A) = 0$.
\end{definition}

\begin{definition}
Let $f \colon X \longrightarrow X$ be a finite flat surjective morphism on a nice variety $X$ 
defined over a field of characteristic zero.
For a closed subscheme $Y \subset X$, we define 
\begin{align}
e(f;Y) := \lim_{n \to \infty} \sup\{ e_{f^{n}}(z) \mid z \in f^{-n}(Y) \}^{1/n}.
\end{align}
See, for example, \cite[Theorem 4.8]{mat2023growthlocal} for the existence of the limit
and basic properties of this quantity.
\end{definition}
Note that in this situation, $X_{f}^{\rm back} = X$ and we actually have
\begin{align}
e(f; Y) = \max_{y \in Y} e_{f,-}(y).
\end{align}
See \cite[Theorem 4.8]{mat2023growthlocal}.

\begin{theorem}[\cref{intro:thm:grwothlocalhtBD}]\label{thm:grwothlocalhtBD1}
Let $K$ be a number field, $X$ a nice variety over $K$, $f \colon X \longrightarrow X$ a surjective morphism.
Let $S \subset M_{K}$ be a finite set of places and $h_{H}$ be a height function associated with an ample divisor $H$ on $X$.
Let $Y \subset X$ be a proper closed subscheme and fix a local height function $\{ \l_{Y,v} \}_{v \in M_{K}}$ associated with $Y$.
Let $x \in X(K)$ be a point.
Suppose
\begin{enumerate}
\item $Y$ contains no positive dimensional $f$-periodic subvariety;
\item $e(f;Y) < \alpha_{f}(x)$.
\end{enumerate}
Then if an infinite sequence $\{n_{j}\}_{j=1}^{\infty}$ satisfies
\begin{align}
\liminf_{j\to \infty} \frac{\sum_{v \in S}\l_{Y,v}(f^{n_{j}}(x))}{h_{H}(f^{n_{j}}(x))} > 0,
\end{align}
then it has Banach density zero.
\end{theorem}

\begin{proof}
First, we may assume $S$ is a singleton.
Indeed, if 
\[
\liminf_{j\to \infty} \frac{\sum_{v \in S}\l_{Y,v}(f^{n_{j}}(x))}{h_{H}(f^{n_{j}}(x))} > 0,
\]
then there is $j_{0} > 0$ and a positive number $c>0$ such that
\[
\frac{\sum_{v \in S}\l_{Y,v}(f^{n_{j}}(x))}{h_{H}(f^{n_{j}}(x))} \geq c
\]
for all $j \geq j_{0}$.
For $v \in S$, set
\[
A_{v} = \left\{ n_{j}\ \middle|\ j \geq j_{0}, \frac{\l_{Y,v}(f^{n_{j}}(x))}{h_{H}(f^{n_{j}}(x))} \geq \frac{c}{\sharp S} \right\}.
\]
Then $\{n_{j}\}_{j=1}^{\infty} = \{n_{j}\}_{j=1}^{j_{0}-1} \cup \bigcup_{v \in S}A_{v}$.
The singleton case of the theorem says $A_{v}$ has Banach density zero, and hence so does $\{n_{j}\}_{j=1}^{\infty} $.

Now we assume $S$ is a singleton, say $S = \{v\}$.
We proceed by induction on $\dim Y$.
If $Y =  \emptyset$, the statement is trivial.
If $\dim Y = 0$, then the statement follows from \cite[Theorem 1.11]{mat2023growthlocal}.

Suppose $\dim Y \geq 1$.
First, note that we may assume $Y$ is irreducible.
Indeed, let $Y = Y_{1} \cup \cdots \cup Y_{r}$ be the irreducible decomposition.
We equip $Y_{i}$ the reduced structure and fix local heights $\l_{Y_{i}, v}$.
Then there are constants $a , b> 0$ such that $a \sum_{i=1}^{r}\l_{Y_{i},v} + b\geq \l_{Y,v}$.
If 
\[
\liminf_{j\to \infty} \frac{\l_{Y,v}(f^{n_{j}}(x))}{h_{H}(f^{n_{j}}(x))} > 0,
\]
then there is $j_{0} > 0$ and a positive number $c>0$ such that
\[
\frac{\l_{Y,v}(f^{n_{j}}(x))}{h_{H}(f^{n_{j}}(x))} \geq c
\]
for all $j \geq j_{0}$.
This implies 
\begin{align}
\frac{a \sum_{i=1}^{r}\l_{Y_{i},v}(f^{n_{j}}(x)) + b}{h_{H}(f^{n_{j}}(x))} \geq \frac{\l_{Y,v}(f^{n_{j}}(x))}{h_{H}(f^{n_{j}}(x))} \geq c.
\end{align}
By the assumption, we have $ \alpha_{f}(x) > 1$ and this implies $h_{H}(f^{n_{j}}(x))$ goes to infinity.
Thus there are $j_{1} \geq j_{0}$ and $c' > 0$ such that
\[
\frac{\sum_{i=1}^{r}\l_{Y_{i},v}(f^{n_{j}}(x))}{h_{H}(f^{n_{j}}(x))} \geq c'
\]
for all $j \geq j_{1}$.
Set 
\[
B_{i} = \left\{n_{j}\   \middle| \ j \geq j_{1},\ \frac{\l_{Y_{i},v}(f^{n_{j}}(x))}{h_{H}(f^{n_{j}}(x))} \geq \frac{c'}{r}   \right\}.
\]
Then $\{n_{j}\}_{j \geq j_{1}} = \bigcup_{i=1}^{r}B_{i}$. 
The theorem for each $Y_{i}$ implies each $B_{i}$ has Banach density zero.
(We can apply the theorem to $Y_{i}$ because $e(f; Y_{i}) \leq e(f ; Y)$ and $Y_{i}$ contains no positive dimensional $f$-periodic subvariety.)
Hence $\{n_{j}\}_{j\geq 1}$ is a finite union of sets with Banach density zero, which is again has Banach density zero.

Now assume $Y$ is irreducible.
By the assumption, there is $j_{0} \geq 1$ and $\eta >0$ such that
\[
\l_{Y,v}(f^{n_{j}}(x)) \geq \eta h_{H}(f^{n_{j}}(x))
\]
for all $j \geq j_{0}$.
Let us write $P = \{n_{j}\}_{j \geq j_{0}}$.
It is enough to show that $P$ has Banach density zero.
Suppose $P$ has positive Banach density.
Then by \cite[Lemma 2.1]{BGT15DMLnoe},
there are $Q \subset P$ and $k \in \Z_{>0}$ such that
\begin{align}
&\text{$Q$ has positive Banach density;}\\
& n \in Q \Longrightarrow n+k \in P.
\end{align}

By the functoriality of local height functions, there is a constant $C >0$ depending only on
$X, f, k$, and $Y$ such that $\l_{f^{-k}(Y),v} \geq \l_{Y,v} \circ f^{k}- C$.
Hence for all $n \in Q$, we have
\begin{align}
\l_{f^{-k}(Y),v}(f^{n}(x)) \geq \l_{Y,v}(f^{n+k}(x))- C \geq \eta h_{H}(f^{n+k}(x)) - C.
\end{align}
We fix a local height associated with $Y \cap f^{-k}(Y)$ so that
 $\l_{Y \cap f^{-k}(Y), v} = \min\{ \l_{Y,v}, \l_{f^{-k}(Y),v}\}$.
Then for all $n \in Q$,
\begin{align}
\l_{Y \cap f^{-k}(Y), v}(f^{n}(x)) &= \min\{ \l_{Y,v}(f^{n}(x)), \l_{f^{-k}(Y),v}(f^{n}(x))  \}\\
& \geq \min\{ \eta h_{H}(f^{n}(x)) , \eta h_{H}(f^{n+k}(x)) - C  \}.
\end{align}
Dividing by $h_{H}(f^{n}(x))$, we get
\begin{align}
\frac{\l_{Y \cap f^{-k}(Y), v}(f^{n}(x)) }{h_{H}(f^{n}(x))}\geq \min\left\{ \eta  , \eta \frac{h_{H}(f^{n+k}(x)) }{h_{H}(f^{n}(x))} - \frac{C}{h_{H}(f^{n}(x))}  \right\}.
\end{align}
Noticing that $h_{H}(f^{n+k}(x)) = h_{(f^{k})^{*}H}(f^{n}(x)) + O(1)$, $(f^{k})^{*}H$ is ample, and $h_{H}(f^{n}(x))$
goes to infinity, we can see
\[
\liminf_{n\in Q} \min\left\{ \eta  , \eta \frac{h_{H}(f^{n+k}(x)) }{h_{H}(f^{n}(x))} - \frac{C}{h_{H}(f^{n}(x))}  \right\} > 0.
\]
On the other hand, by the assumption, $Y$ is not $f$-periodic.
Since $Y$ is irreducible, we get $\dim (Y \cap f^{-k}(Y) )< \dim Y$.
As $Y \cap f^{-k}(Y) \subset Y$,  $Y \cap f^{-k}(Y)$ does not contain positive dimensional $f$-periodic subvariety.
We also have $e(f; Y \cap f^{-k}(Y) ) \leq e(f;Y) < \alpha_{f}(x)$.
Applying induction hypothesis to $Y \cap f^{-k}(Y) $, we get
\[
\liminf_{n\in Q} \frac{\l_{Y \cap f^{-k}(Y), v}(f^{n}(x)) }{h_{H}(f^{n}(x))} =0.
\]
This is contradiction.
\end{proof}

\begin{corollary}
Let $f \colon \P^{N}_{\Q} \longrightarrow \P^{N}_{\Q}$ be a surjective morphism with $d = d_{1}(f) >1$.
Let $x_{0}, \dots, x_{N}$ be the homogeneous coordinates of $\P^{N}_{\Q}$ and $H = (x_{N}=0)$ be a coordinate hyperplane.
Suppose 
\begin{enumerate}
\item $H$ contains no positive dimensional $f$-periodic subvariety;
\item $e(f;H) < d$.
\end{enumerate}
Let $a \in \P^{N}(\Q)$ be a point with infinite $f$-orbit.
Write $f^{n}(a) = (a_{0}(n): \cdots : a_{N}(n))$ where $a_{0}(n),\dots, a_{N}(n)$ are coprime integers.
Then if a sequence $\{n_{j}\}_{j=1}^{\infty}$ satisfies 
\[
\limsup_{j \to \infty} \frac{\log |a_{N}(n_{j})|}{\log \max\{ |a_{0}(n_{j})|,\dots, |a_{N}(n_{j})| \}} < 1,
\]
then $\{n_{j}\}_{j=1}^{\infty}$ has Banach density zero.
Here $|\ |$ is the usual absolute value on $\Q$.
\end{corollary}
\begin{proof}
Let $\infty \in M_{\Q}$ be the infinite place.
Let us fix local and global height functions as follows:
for coprime integers $a_{0},\dots,a_{N}$,
\begin{align}
\l_{H,\infty}(a_{0}:\cdots :a_{N}) &= \log \frac{ \max\{ |a_{0}|,\dots, |a_{N}| \}}{|a_{N}|}\\
h_{H}(a_{0}:\cdots :a_{N}) &= \log \max\{ |a_{0}|,\dots, |a_{N}| \}.
\end{align}

Now, first note that we have $ \alpha_{f}(a) = d  > e(f;H)$ since $a$ has infinite $f$-orbit.
We have
\begin{align}
\frac{\l_{H,\infty}(f^{n}(a))}{h_{H}(f^{n}(a))} = 1 - \frac{\log |a_{N}(n)|}{ \log \max\{ |a_{0}(n)|,\dots, |a_{N}(n)| \}}.
\end{align}
Hence 
\begin{align}
\limsup_{j \to \infty} \frac{\log |a_{N}(n_{j})|}{\log \max\{ |a_{0}(n_{j})|,\dots, |a_{N}(n_{j})| \}} < 1
\end{align}
is equivalent to
\begin{align}
\liminf_{j \to \infty} \frac{\l_{H,\infty}(f^{n_{j}}(a))}{h_{H}(f^{n_{j}}(a))}  > 0.
\end{align}
By \cref{thm:grwothlocalhtBD1}, we get the statement.
\end{proof}

\appendix

\section{Roth's theorem}

In this appendix, we reformulate Roth's theorem using local height function associated with subschemes.
For a $0$-dimensional closed subscheme $Y \subset X$ of an algebraic scheme $X$, we set 
\begin{align}
m_{X}(Y) = \max\big\{ \min\{i \mid \m_{X,y}^{i} \subset \I_{Y,y} \} \mid y \in Y\big\}
\end{align}
where $\m_{X, y}$ is the maximal ideal of the local ring $\O_{X, y}$
and $\I_{Y}$ is the ideal sheaf of $Y \subset X$.

\begin{theorem}[Roth's theorem]\label{thm:roth}
Let $K$ be a number field.
Let $X$ be a reduced projective scheme over $K$ such that every irreducible component is positive dimensional,
$Y \subset X$ a closed subscheme with $\dim Y=0$.
Let $H$ be a very ample divisor on $X$.
Fix local height function $\{ \l_{Y,v} \}_{v \in M_{K}}$ associated with $Y$,
and Weil height function $h_{H}$ associated with $H$.
Then for any finite subset $S \subset M_{K}$ and $\e>0$, there is a finite subset $Z \subset X(K)$
such that 
\begin{align}
\sum_{v \in S} \l_{Y,v}(x) \leq m_{X}(Y) (2+\e) h_{H}(x)
\end{align}
for all $x \in X(K) \setminus Z$.
\end{theorem}
\begin{proof}
Since
\begin{align}
\sum_{v \in S} \l_{Y,v} \leq m_{X}(Y) \sum_{v \in S} \l_{Y_{\rm red},v} + O(1),
\end{align}
we may assume $Y$ is reduced.
Let $K \subset L$ be a finite extension such that all the residue fields of $Y_{L}$ are $L$.
By replacing $K$ with $L$, $Y$ with $Y_{L}$, $X$ with $X_{L}$, $H$ with $H_{L}$,
and $S$ with $\left\{ w \in M_{L} \ \middle|\  w | v\right\}$, we may assume the  residue field of every point of $Y$ is $K$.

Let $ \Delta \subset X \times X$ be the diagonal.
For any $v \in M_{K}$, $y_{1}\neq y_{2} \in Y$, and $x \in X( K)$, we have
\begin{align}
\min\{ \l_{y_{1}, v}(x), \l_{y_{2}, v}(x) \} &\leq \min\{ \l_{ \Delta, v}(y_{1}, x) , \l_{ \Delta, v}(y_{2}, x) \} + \g_{y_{1}, y_{2}}(v)\\
& \leq \l_{\pr_{13}^{-1}( \Delta) \cap \pr_{23}^{-1}( \Delta) , v}(y_{1}, y_{2}, x) + \g'_{y_{1},y_{2}}(v)\\
& \leq \l_{\pr_{12}^{-1}( \Delta), v} (y_{1}, y_{2}, x) + \g''_{y_{1},y_{2}}(v)\\
& \leq \l_{\Delta, v} (y_{1}, y_{2}) + \g'''_{y_{1},y_{2}}(v)
\end{align}
where $\g_{y_{1}, y_{2}}(v)$'s are $M_{K}$-constants, i.e.\ constants such that $=0$ for all but finitely many $v \in M_{K}$.  
Thus
\begin{align}
\sum_{v \in S}\l_{Y,v}(x) &=  \sum_{v \in S} \sum_{y \in Y}\l_{y,v}(x) + O(1) \\
& \leq \sum_{v \in S} \max\{ \l_{y, v}(x) \mid y \in Y\} + O(1)
\end{align}
where the implicit constants depend at most on $Y$, $X$, and the choice of all the local height functions appeared so far.
Therefore it is enough to prove the following.
Let $S \subset M_{K}$ be a  finite subset.  Suppose $y_{v} \in X(K)$ are given  for each $v \in S$.
Then 
\begin{align}
\sum_{v \in S} \l_{y_{v}, v}(x) \leq (2 + \e)h_{H}(x)
\end{align}
for all but finitely many $x \in X(K)$.
By taking closed immersion $X \hookrightarrow \P^{N}_{K}$ defined by
general members of the linear system $|H|$ and replacing $X$ with $\P^{N}_{K}$, 
we may assume $X = \P^{N}_{K}$, $h_{H} = h_{\P^{N}}$, the naive height on $\P^{N}$, and $y_{v} \in D_{+}(t_{0})$ for all $v \in S$.
Here $t_{0}, \dots, t_{N}$ are the homogeneous coordinates of $\P^{N}_{K}$.

Let $p_{i} \colon D_{+}(t_{0}) = \Spec K[ \frac{t_{1}}{t_{0}}, \dots, \frac{t_{N}}{t_{0}} ] \longrightarrow \Spec K[ \frac{t_{i}}{t_{0}}]$
be the $i$-th projection.
Let $y_{v} = (y_{v}^{(1)}, \dots, y_{v}^{(N)})$ be the affine coordinates of $y_{v}$ as a $K$-point of $D_{+}(t_{0})$,
i.e.\ $y_{v}^{(i)} = p_{i}(y_{v})$.
Then take a bounded neighborhood $B_{v} \subset D_{+}(t_{0})(K)$ of $y_{v}$ with respect to $v$ 
and constant $C \geq 0$
so that
\begin{align}
\l_{y_{v},v}(x) \leq \min_{1 \leq i \leq N}\bigg\{ \log \frac{1}{ |y_{v}^{(i)} - x^{(i)}  |_{v} } \bigg\} + C
\end{align}
for $x \in B_{v} \setminus \{ y_{v}\}$, where $x^{(i)}$ are affine coordinates of $x$.
Note that $\l_{y_{v} ,v}$ is bounded on $\P^{N}(K) \setminus B_{v}$. 
We set 
\begin{align}
C' = \max_{v \in S} \sup_{x \in \P^{N}(K) \setminus B_{v}} \l_{y_{v} , v}(x).
\end{align}

Thus
\begin{align}
\sum_{v \in S} \l_{y_{v}, v}(x)
 \leq  
\min_{1 \leq i \leq N} \bigg\{ \sum_{\substack{v \in S \\ x \in B_{v}}} \log \frac{1}{ |y_{v}^{(i)} - x^{(i)}  |_{v} } \bigg\} + \sharp S C'.
\end{align}
By Roth's theorem (cf.\ \cite[Theorem 6.4.1]{BG06}), for any $\e > 0$, there are finite sets $W_{i} \subset K$ such that
\begin{align}
\sum_{\substack{v \in S \\ x \in B_{v}}} \log \frac{1}{ |y_{v}^{(i)} - x^{(i)}  |_{v} } \leq (2+ \e) h(x^{(i)})
\end{align}
if $x^{(i)} \notin W_{i}$. Here $h(x^{(i)})$ is the logarithmic Weil height of algebraic number $x^{(i)}$.
Since 
\begin{align}
h(x^{(i)})& = \sum_{v \in M_{K}} \log \max\{ 1, |x^{(i)}|_{v}\} \\
& \leq \sum_{v \in M_{K}} \log \max\{ 1, |x^{(1)}|_{v}, \dots, |x^{(N)}|_{v}\} = h_{\P^{N}}(x),
\end{align}
we get
\begin{align}
\sum_{v \in S} \l_{y_{v}, v}(x) \leq (2 + \e) h_{\P^{N}}(x) + \sharp S C'
\end{align}
for $x \in \P^{N}(K) \setminus \bigcap_{i=1}^{N}p_{i}^{-1}(W_{i})$.
Adjusting $\e$ and enlarging the exceptional set, we can remove the constant $\sharp S C'$ and we are done.

\end{proof}

\section{Backward limit of upper semicontinuous functions under rational maps}\label{section:uscbackconvergence}

In this section, we define the locus where the backward iteration behave well
and prove the existence of certain limit along backward orbits, which directly implies the existence of the limit
defining ``$e(f;Y)$".
For a self-map on Noetherian topological space, what we want is exactly the one established in \cite{analyticmult-dinh}.
As we need the same statement for a slightly generalized setting where we can only consider the backward iteration,
we include here the self-contained proof.
The argument is exactly the same with that in \cite{analyticmult-dinh}, 
but with slight changes in basic settings.

\subsection{Finite loci of rational map}\label{subsec:finlociratmap}

Let $K$ be a field, $X$ a projective variety over $K$, and $f \colon X \dashrightarrow X$ a dominant rational map.

\begin{definition}\label{def:maxopenfinite}
Let
\begin{equation}
\begin{tikzcd}
\G_{f} \arrow[d,"p_{1}",swap] \arrow[dr, "p_{2}"]& \\
X \arrow[r, "f", swap,dashed]& X
\end{tikzcd}
\end{equation}
where $\G_{f}$ is the graph of $f$.
Let $U \subset X$ be the largest open subset such that
\begin{align}
&p_{2} \colon p_{2}^{-1}(U) \longrightarrow U \ \text{is finite and}\\
&p_{2}^{-1}(U) \cap p_{1}^{-1}(I_{f}) =  \emptyset.
\end{align}
We call $U$ the largest open subset over which $f$ is finite.
\end{definition}
This definition is justified by the following lemma.

\begin{lemma}\label{lem:Ujustification}
Let $U \subset X$ be as in \cref{def:maxopenfinite}.
If $V,W \subset X$ are open subset such that $V \cap I_{f} =  \emptyset$ and $f|_{V} \colon V \longrightarrow W$ is finite,
then $W \subset U$.
\end{lemma}
\begin{proof}
Consider the following diagram:
\begin{equation}
\begin{tikzcd}
&p_{2}^{-1}(W) \arrow[d, hook] \arrow[rddd, bend left=70, " \beta"] & \\
p_{1}^{-1}(V) \arrow[rdd, " \alpha"] \arrow[r, hook] \arrow[ru, hook] & \G_{f} \arrow[d,"p_{1}",swap] \arrow[dr, "p_{2}"]& \\
&X \arrow[r, "f", swap,dashed]& X\\
&V \arrow[u, phantom, "\cup"]  \arrow[r,  "f|_{V}"]  & W \arrow[u, phantom, "\cup"]  
\end{tikzcd}
\end{equation}
Since $f|_{V} \circ \alpha$ and $ \beta$ are proper, $p_{1}^{-1}(V) = p_{2}^{-1}(W)$.
As $V \cap I_{f} =  \emptyset$, $ \alpha$ is isomorphism and hence $ \beta$ is finite and $p_{2}^{-1}(W) \cap p_{1}^{-1}(I_{f})=  \emptyset$.
Thus $W \subset U$.
\end{proof}
When we say $f$ is finite over an open subset $W \subset X$, this simply means $W \subset U$.

We introduce open subsets $U_{n} \subset X$ over which $f$ is repeatedly finite at least $n$ times.
\begin{definition}
We define a sequence of open subsets $U_{n} \subset X$, $n=0,1,2, \dots$ inductively as follows.
Let $U \subset X$ be as in \cref{def:maxopenfinite}.
We set $U_{0} = X$.
Let $n \geq 0$ and suppose we have constructed $U_{n}$  so that
$f^{n}$ is finite over $U_{n}$.
Then define 
\begin{align}
U_{n+1} = U_{n} \setminus f^{n}( f^{-n}(U_{n}) \setminus U ).
\end{align}
The following diagram summarizes the situation.
\begin{equation}
\begin{tikzcd}
X\arrow[r, "f", dashed]  &X \arrow[r, "f^{n}", dashed] & X \\
&f^{-n}(U_{n}) \arrow[r] \arrow[u, phantom, "\cup"] & U_{n} \arrow[u, phantom, "\cup"] \\
&f^{-n}(U_{n}) \cap U \arrow[u, phantom, "\cup"] & U_{n+1} \arrow[u, phantom, "\cup"] \\
f^{-n-1}(U_{n+1}) \arrow[r] \arrow[uuu, hook] &f^{-n}(U_{n+1}) \arrow[u, phantom, "\cup"] \arrow[ru] &
\end{tikzcd}
\end{equation}
\end{definition}
By the construction, for every $n \geq 0$, we have
\begin{equation}
\begin{tikzcd}
X \arrow[r, "f", dashed] & X \arrow[r, "f", dashed]  & \cdots \arrow[r, "f", dashed]  & X \arrow[r, "f", dashed]  & X \\
f^{-n}(U_{n}) \arrow[u, phantom, "\subset", sloped] \arrow[r, "f"] & f^{-(n-1)}(U_{n}) \arrow[u, phantom, "\subset", sloped]  \arrow[r, "f"]  & \cdots  \arrow[r, "f"] & f^{-1}(U_{n}) \arrow[u, phantom, "\subset", sloped] \arrow[r, "f"] & U_{n}  \arrow[u, phantom, "\subset", sloped] \\
& U \arrow[u,phantom, "\supset", sloped]  && U \arrow[u,phantom, "\supset", sloped]  & U \arrow[u,phantom, "\supset", sloped]  
\end{tikzcd}
\end{equation}
Note that all the morphisms in the second row are finite.
Actually, $U_{n}$ is the largest open subset with this property.

\begin{lemma}\label{lem:maximalityUn}
Let $W_{0}, \dots, W_{n} \subset X$ be open subsets such that
$W_{i} \cap I_{f} =  \emptyset$ and $f|_{W_{i}} \colon W_{i} \longrightarrow W_{i-1}$ are finite for $i = 1, \dots , n$, i.e.
\begin{equation}
\begin{tikzcd}
X \arrow[r, "f", dashed] & X \arrow[r, "f", dashed]  & \cdots \arrow[r, "f", dashed]  & X \arrow[r, "f", dashed]  & X \\
W_{n} \arrow[u, phantom, "\subset", sloped] \arrow[r, "\text{finite}"] & W_{n-1}\arrow[u, phantom, "\subset", sloped]  \arrow[r, "\text{finite}"]  & \cdots  \arrow[r, "\text{finite}"] & W_{1}\arrow[u, phantom, "\subset", sloped] \arrow[r, "\text{finite}"] & W_{0}  \arrow[u, phantom, "\subset", sloped] \\
\end{tikzcd}
\end{equation}
Then we have $W_{0} \subset U_{n}$.
\end{lemma}
\begin{proof}
We prove by induction on $n$.
If $n=0$, $U_{0}= X$ and the statement is trivial.
Let $n \geq 0$ and suppose the statement holds for $n$.
Then we have the following diagram.
\begin{equation}
\begin{tikzcd}
X  \arrow[r, "f", dashed]  & X \arrow[r, "f^{n}", dashed]  & X \\
& f^{-n}(U_{n}) \arrow[u, phantom, "\subset", sloped] \arrow[r, "f^{n}"]  & U_{n} \arrow[u, phantom, "\subset", sloped] \\
W_{n+1} \arrow[r] \arrow[uu, hook]& W_{n} \arrow[u, phantom, "\subset", sloped] \arrow[r] & W_{0} \arrow[u, phantom, "\subset", sloped]
\end{tikzcd}
\end{equation}
By \cref{lem:Ujustification}, $f$ is finite over $W_{n}$, that is $W_{n} \subset U$.
Also, since $f^{n} \colon f^{-n}(U_{n}) \longrightarrow U_{n}$ and $W_{n} \longrightarrow W_{0}$
are finite, we have $W_{n} = f^{-n}(W_{0})$.
Thus 
\begin{align}
W_{0} \subset U_{n} \setminus f^{n}(f^{-n}(U_{n}) \setminus U) = U_{n+1}.
\end{align}
\end{proof}

\begin{corollary}
We have
\begin{align}
&X = U_{0} \supset U_{1} \supset U_{2} \cdots\\
&X = f^{-0}(U_{0}) \supset f^{-1}(U_{1}) \supset f^{-2}(U_{2}) \cdots
\end{align}

\end{corollary}

\begin{definition}
We set
\begin{align}
U_{\infty} = \bigcap_{n=0}^{\infty} U_{n}.
\end{align}
We equip the subspace topology from $X$.
With this topology, $U_{\infty}$ is a Zariski topological space (i.e. Noetherian topological space such that
every irreducible closed subset has unique generic point).
We sometimes write $U_{\infty} = X_{f}^{\rm back}$.
\end{definition}
We have $f^{-1}(U_{\infty}) \subset U_{\infty}$.
Indeed, for every $n \geq 0$, we have
\begin{align}
f^{-1}(U_{n+1}) \subset U_{n}.
\end{align}

\subsection{Backward limit of upper semicontinuous functions}

We keep the notation from the previous subsection.
\begin{definition}
A sequence $\k_{n} \colon f^{-n}(U_{n}) \longrightarrow [1, \infty), n=0,1,2,\dots$ of upper semicontinuous functions 
is called submultiplicative cocycle if
\begin{enumerate}
\item $\min\{ \k_{n}(x) \mid x \in  f^{-n}(U_{n})  \}=1$ for all $n \geq 0$;
\item for all $m,n \geq 0$ and $x \in f^{-m-n}(U_{m+n})$, we have
\begin{align}
\k_{m+n}(x) \leq \k_{n}(x) \k_{m}(f^{n}(x)).
\end{align}
Here note that the right hand side is well-defined since 
$x \in f^{-m-n}(U_{m+n}) \subset f^{-n}(U_{n})$ and $f^{n}(x) \in f^{-m}(U_{m+n}) \subset f^{-m}(U_{m})$.
\end{enumerate}
\end{definition}

We note that any upper semicontinuous function on a Zariski topological space 
attains maximum at some point and minimum at some generic point.

For a submultiplicative cocycle $\{\k_{n}\}_{n \geq 0}$, we define 
\begin{align}
\k_{-n} \colon U_{n} \longrightarrow \R, x \mapsto \max_{y \in f^{-n}(x)} \k_{n}(y).
\end{align}
Here $f^{-n}(x)$ is the inverse image of $x$ by the finite morphism
\begin{align}
f^{-n}(U_{n}) \longrightarrow U_{n}.
\end{align}
For every $\d \in \R$, we have
\begin{align}
\{ x \in U_{n} \mid \k_{-n}(x) \geq \d \} = f^{n}(\{ y \in f^{-n}(U_{n}) \mid \k_{n}(y) \geq \d \} ).
\end{align}
Since $f^{n} \colon f^{-n}(U_{n}) \longrightarrow U_{n}$ is a closed map, this set is closed in $U_{n}$.
Thus $\k_{-n}$ is upper semicontinuous on $U_{n}$.

The following is the goal of this section.
\begin{theorem}\label{thm:kappanegative}
For $x \in U_{\infty}$, the limit
\begin{align}
\lim_{n \to \infty} \k_{-n}(x)^{ \frac{1}{n}} =: \k_{-}(x)
\end{align}
exists.
Moreover, for any $\d \in \R_{>1}$, the set
\begin{align}
\{ x \in U_{\infty} \mid \k_{-}(x) \geq \d \}
\end{align}
is a proper closed subset of $U_{\infty}$.
In particular, $\k_{-} \colon U_{\infty} \longrightarrow \R$ is upper semicontinuous.
\end{theorem}

First we note that
\begin{align}
1 \leq \liminf_{n \to \infty} \k_{-n}(x)^{ \frac{1}{n}} \leq \limsup_{n \to \infty} \k_{-n}(x)^{ \frac{1}{n}} < \infty
\end{align}
for all $x \in U_{\infty}$.
Indeed, by definition we have $\k_{-n} \geq 1$ and thus the first inequality is trivial.
For the last inequality, note that
\begin{align}
\k_{-n}(x) &= \max_{y \in f^{-n}(x)} \k_{n}(y) \leq \max_{y \in f^{-n}(x)} \k_{1}(y) \k_{1}(f(y)) \cdots \k_{1}(f^{n-1}(y))\\
& \leq \big( \max_{z \in f^{-1}(U_{1})}\k_{1}(z) \big)^{n}.
\end{align}
Thus we get $\limsup_{n \to \infty} \k_{-n}(x)^{ \frac{1}{n}} \leq \max_{z \in f^{-1}(U_{1})}\k_{1}(z)  < \infty$.

\begin{lemma}\label{lem:limsupinf-divisible}
For all $x \in U_{\infty}$ and $l \geq 1$, we have
\begin{align}
&\limsup_{n \to \infty} \k_{-nl}(x)^{ \frac{1}{nl}} = \limsup_{n \to \infty} \k_{-n}(x)^{ \frac{1}{n}}\\
&\liminf_{n \to \infty} \k_{-nl}(x)^{ \frac{1}{nl}} = \liminf_{n \to \infty} \k_{-n}(x)^{ \frac{1}{n}}.
\end{align}
\end{lemma}
\begin{proof}
Let $M = \max_{0 \leq s \leq l} \max_{z \in f^{-s}(U_{s})} \k_{s}(z)$.
For any $m \geq 0$, $s \in \{0, \dots, l\}$, take $y \in f^{-m-s}(x)$ such that
$\k_{-m-s}(x) = \k_{m+s}(y)$.
Then 
\begin{align}
\k_{-m-s}(x) = \k_{m+s}(y) \leq \k_{s}(y) \k_{m}(f^{s}(y)) \leq M \k_{-m}(x).
\end{align}
Thus
\begin{align}
M^{-1}\k_{-(n+1)l}(x) \leq \k_{-nl -s}(x) \leq M\k_{-nl}(x)
\end{align}
for all $n \geq 0$ and $s \in \{0, \dots, l\}$. Thus we are done.
\end{proof}

As we are working on a locus where only the backward orbits
are well-defined, we need the following non-standard definition of periodic points.
\begin{definition}
Let $y \in U_{\infty}$.
We say $y$ is $f$-periodic if 
\begin{align}
y \in f^{-l}(y)
\end{align}
for some $l \in \Z_{\geq 1}$.
\end{definition}

\begin{lemma}\label{lem:basic-per-pt}
Let $y \in U_{\infty}$ be an $f$-periodic point.
Then
\begin{enumerate}
\item
$ y \in f^{-m}(U_{m})$ for all $m \geq 0$;
\item
$f^{m}(y)$ is also $f$-periodic for all $m \geq 0$.
\end{enumerate}
\end{lemma}
\begin{proof}
Suppose $y \in f^{-l}(y)$ with $l \geq 1$.
Then $y \in f^{-nl}(y) \subset f^{-nl}(U_{\infty})$
for all $n \geq 0$.
Then for any $n \geq 0$ and $s \in \{0 ,\dots , l-1\}$, we have
\begin{align}
y \in f^{-(n+1)l}(U_{\infty}) \subset f^{-(n+1)l}(U_{(n+1)l})  \subset f^{-nl -s}(U_{nl+s}).
\end{align}
This proves (1).

Next, let $m \geq 0$.
Take $n \geq 0$ such that $nl \geq m$.
Since $y \in f^{-nl}(y)$, we have $f^{nl-m}(f^{m}(y))=y$.
Thus $f^{m}(y) \in f^{-(nl-m)}(y) \subset f^{-(nl-m)}(U_{\infty}) \subset U_{\infty}$.
Moreover, since $y \in f^{-nl-m}(U_{nl+m})$, $f^{nl}(f^{m}(y)) = f^{m}(y)$.
This proves (2).
\end{proof}

\begin{lemma}\label{lem:perkappa}
Let $y \in U_{\infty}$ be an $f$-periodic point.
Then $\k_{n}(y)$ is well-defined for all $n \geq 0$ by \cref{lem:basic-per-pt},
and we have the following.
\begin{enumerate}
\item
$\k_{+}(y) := \lim_{n \to \infty} \k_{n}(y)^{ \frac{1}{n}}$ exists.

\item
For all $m \geq 0$, we have $\k_{+}(f^{m}(y)) = \k_{+}(y)$.

\item
$ \liminf_{n \to \infty} \k_{-n}(y)^{ \frac{1}{n}} \geq \k_{+}(y)$.

\end{enumerate}
\end{lemma}
\begin{proof}
Let $l \geq 1$ be such that $y \in f^{-l}(y)$.
Then we have
\begin{align}
\k_{(m+n)l}(y) \leq \k_{ml}(y) \k_{nl}(f^{ml}(y)) = \k_{ml}(y) \k_{nl}(y).
\end{align}
Thus by Fekete's lemma, $\lim_{n \to \infty} \k_{nl}(y)^{ \frac{1}{nl}}$ exists.
For $m \geq0$ and $s \in \{0, \dots, l\}$, we have
\begin{align}
\k_{m+s}(y) \leq \k_{m}(y) \k_{s}(f^{m}(y)) \leq M \k_{m}(y)
\end{align}
where $M = \max_{0 \leq s \leq l} \max_{z \in f^{-s}(U_{s})} \k_{s}(z)$.
Thus 
\begin{align}
M^{-1} \k_{(n+1)l}(y) \leq \k_{nl + s} (y) \leq M \k_{nl}(y).
\end{align}
Thus $ \lim_{n \to \infty} \k_{n}(y)^{ \frac{1}{n}}$ exists.

Similarly, we have
\begin{align}
&\k_{n+1}(y) \leq \k_{1}(y) \k_{n}(f(y)) \leq M \k_{n}(f(y))\\
&\k_{n}(f(y)) \leq \k_{l-1}(f(y)) \k_{n-(l-1)}(f^{l}(y)) \leq M \k_{n-(l-1)}(y).
\end{align}
Thus we get $\k_{+}(f(y)) = \k_{+}(y)$. Thus inductively we get $\k_{+}(f^{m}(y)) = \k_{+}(y)$ for all $m \geq 0$.

Finally, since $\k_{-nl}(y) = \max_{z \in f^{-nl}(y)} \k_{nl}(z) \geq \k_{nl}(y)$, we have
\begin{align}
\liminf_{n \to \infty} \k_{-nl}(y)^{ \frac{1}{nl}} \geq \lim_{n \to \infty} \k_{nl}(y)^{ \frac{1}{nl}} = \k_{+}(y).
\end{align}
By \cref{lem:limsupinf-divisible}, we are done.
\end{proof}

\begin{lemma}
Let $T$ be a topological space.
Let $\{O_{\l}\}_{\l \in \L}$ be a family of open subsets of $T$.
Let $Y \subset T$ be an irreducible closed subset with generic point $\eta \in Y$.
Then $Y \cap \bigcap_{\l \in \L} O_{\l} \neq  \emptyset$ if and only if $\eta \in \bigcap_{\l \in \L} O_{\l}$.
\end{lemma}
\begin{proof}
This is obvious.
\end{proof}
As $U_{\infty}$ is intersection of open subsets of $X$, this lemma applies to
any open subset of $U_{\infty}$ and irreducible closed subset of $X$.
That is, for any open subset $W \subset U_{\infty}$ and irreducible closed subset $Y \subset X$,
$Y \cap W \neq  \emptyset$ if and only if the generic point of $Y$ is contained in $W$.

We also note that an intersection of a family of open subsets of $X$ is a Zariski topological space.
Thus any open subset $W \subset U_{\infty}$ is a Zariski topological space.
For a closed subset $Z \subset W$, let $\eta_{1}, \dots, \eta_{r}$ be the generic points of $Z$.
Then $\eta_{1}, \dots, \eta_{r}$ are also the set of all generic points of $ \overline{Z}$, 
the Zariski closure of $Z$ in $X$.

In the following, all the closures are taken as subsets of $X$.

\begin{proposition}\label{prop:k-induction}
Let $\dim X = k$. Let $q \in \{0, \dots, k-1\}$.
Let $ \Omega \subset U_{\infty}$ be a non-empty open subset such that $f^{-1}( \Omega) \subset \Omega$.
Assume there exist $n_{0} \in \Z_{\geq 1}$ and $\d \in \R_{>1}$ such that
\begin{align}
\dim \overline{ \{ x \in \Omega \mid \k_{-n_{0}}(x) \geq \d^{n_{0}} \} } \leq q.
\end{align}
Then there are $n_{1} \in \Z_{\geq 1}$ and a proper closed subset $ \Sigma \subsetneq U_{\infty}$ such that
\begin{enumerate}
\item
every generic point $\eta \in \Sigma$ is $f$-periodic, and $f^{m}(\eta) \in \Sigma$ for $m \geq 0$;
\item
either $ \Sigma =  \emptyset$, or $ \overline{ \Sigma}$ is of pure dimension $q$;
\item
$\liminf_{n \to \infty} \k_{-n}(x)^{ \frac{1}{n}} \geq \d$ for all $x \in \Sigma$;
\item
$\dim \overline{ \{ x \in \Omega \setminus \Sigma \mid \k_{-n_{1}}(x) \geq \d^{n_{1}} \} } \leq q-1$.
(When $q=0$, this means the set in the left hand side is empty.)
\end{enumerate}
Note that in this case, we have $f^{-1}( \Omega \setminus \Sigma) \subset  \Omega \setminus \Sigma$.
\end{proposition}

\begin{proof}
Let $ \widetilde{ \Sigma}$ be the union of the irreducible components $Z$ of 
$\{ x \in \Omega \mid \k_{-n_{0}}(x) \geq \d^{n_{0}} \}$ such that $\dim \overline{Z} = q$.

\begin{claim}\label{claim:delta0}
There is $1 \leq \d_{0} < \d$ such that for every $x \in f^{-n_{0}}( \Omega \setminus \widetilde{\Sigma})$
with $\dim \overline{\{x\}} \geq q$, we have $\k_{n_{0}}(x) \leq \d_{0}^{n_{0}}$.
\end{claim}
\begin{proof}
Since $\k_{-n_{0}} \colon U_{n_{0}} \longrightarrow \R$ is upper semicontinuous and
$U_{n_{0}}$ is Noetherian topological space, there is $\d_{0} < \d$ such that
\begin{align}
\{ z\in U_{n_{0}} \mid \k_{-n_{0}}(z) \geq \d^{n_{0}} \} = \{ z\in U_{n_{0}} \mid \k_{-n_{0}}(z) > \d_{0}^{n_{0}} \}.
\end{align} 
For $x$ as in the claim, suppose $\k_{n_{0}}(x) > \d_{0}^{n_{0}}$.
Then $\k_{-n_{0}}(f^{n_{0}}(x)) \geq \k_{n_{0}}(x) > \d_{0}^{n_{0}}$, and hence $\k_{-n_{0}}(f^{n_{0}}(x)) \geq \d^{n_{0}}$.
Then by the choice of $ \widetilde{ \Sigma}$, $\dim \overline{ \{f^{n_{0}}(x)\}} \leq q-1$.
But $f^{n_{0}} \colon f^{-n_{0}}(U_{n_{0}}) \longrightarrow U_{n_{0}}$ is finite and $x \in f^{-n_{0}}(U_{n_{0}}) $,
this is a contradiction.
\end{proof}

We define closed subsets $ \Sigma, \Sigma', \Sigma'' \subset U_{\infty}$ as follows.
Let 
\begin{align}
\Sigma = \bigcup \left\{ \overline{\{f^{i}(y)\}} \cap U_{\infty} \ \middle|\ \parbox{15em}{$i \geq 0$, $y$ is an $f$-periodic generic point of $ \widetilde{ \Sigma}$ such that $\k_{+}(y) \geq \d$ } \right\}.
\end{align}
Note that the generic point of $U_{\infty}$ is $f$-periodic and $\k_{+}=1$.
Thus $ \Sigma \subsetneq U_{\infty}$.
It is easy to see that $ \Sigma$ satisfies the properties (1)(2).
Moreover, for any $x \in \Sigma$, take $y$ as in the definition of $ \Sigma$ so that $x \in \overline{f^{i}(y)}$.
Then by the upper semicontinuity of $\k_{-n}$'s on $U_{\infty}$ and \cref{lem:perkappa}, we have
\begin{align}
\liminf_{n \to \infty} \k_{-n}(x)^{ \frac{1}{n}} \geq \liminf_{n \to \infty} \k_{-n}(f^{i}(y))^{ \frac{1}{n}} \geq \k_{+}(f^{i}(y)) = \k_{+}(y) \geq \d.
\end{align}
Thus (3) also holds.
Furthermore, we have $f^{-1}( \Omega \setminus \Sigma) \subset \Omega \setminus \Sigma$.
Indeed, let $x \in f^{-1}( \Omega \setminus \Sigma) $.
Since $f^{-1}( \Omega) \subset \Omega$, we only need to show that $x \not\in \Sigma$.
Suppose $x \in \Sigma$. Take a generic point $y \in \Sigma$ such that $x \in \overline{\{y\}}$.
Since $x, y \in f^{-1}(U_{1})$, we have
\begin{align}
f(x) \in f( \overline{\{y\}} \cap f^{-1}(U_{1})) = \overline{\{f(y)\}} \cap U_{1}.
\end{align}
Since $f(y) \in \Sigma$, we get $f(x) \in \Sigma$ and this is contradiction.

Next, let
\begin{align}
\Sigma' = \bigcup \left\{ \overline{\{f^{i}(y)\}} \cap U_{\infty} \ \middle|\ \parbox{15em}{$i \geq 0$, $y$ is an $f$-periodic generic point of $ \widetilde{ \Sigma}$ such that $\k_{+}(y) < \d$ } \right\}.
\end{align}
Then by \cref{lem:perkappa}, for any generic point $z \in \Sigma' $, we have $\k_{+}(z) < \d$.
Thus there are $1 \leq \d_{1} < \d$ and $m \geq 1$ such that 
\begin{align}
\k_{mn_{0}}(z) \leq \d_{1}^{mn_{0}}
\end{align}
for all generic points $z \in \Sigma'$.

Finally, let
\begin{align}
\Sigma'' = \bigcup \left\{ \overline{\{y\}} \cap U_{\infty} \ \middle|\ \parbox{15em}{$y$ is a generic point of $ \widetilde{ \Sigma}$, which is not $f$-periodic } \right\}.
\end{align}
We set $l$ to be the number of generic points of $\Sigma'' $.

Now we set
\begin{align}
M = \max_{z \in f^{-n_{0}}(U_{n_{0}})} \k_{n_{0}}(z).
\end{align}
Then take $N \geq 1$ so that
\begin{align}
M^{l} \d_{0}^{Nmn_{0}}, \d_{1}^{Nmn_{0}}, M^{l+m} \max\{\d_{0},\d_{1}\}^{Nmn_{0}} < \d^{Nmn_{0}}.
\end{align}
Set $n_{1} = Nmn_{0}$.
We claim that (4) holds for this $n_{1}$.
To this end, it is enough to show that for any $\eta_{0} \in \Omega \setminus \Sigma$ with
$\dim \overline{\{\eta_{0}\}} =: q' \geq q$, we have $ \k_{-n_{1}}(\eta_{0}) < \d^{n_{1}}$.

For arbitrary $\eta' \in f^{-Nmn_{0}}(\eta_{0})$, we set
\begin{equation}
\adjustbox{scale=.8,center}{
\begin{tikzcd}
f^{-Nmn_{0}}(U_{Nmn_{0}}) \arrow[r, "f^{n_{0}}"] & f^{-(Nm -1)n_{0}}(U_{Nmn_{0}}) \arrow[r, "f^{n_{0}}"] & \cdots \arrow[r, "f^{n_{0}}"] & f^{-n_{0}}(U_{Nmn_{0}}) \arrow[r, "f^{n_{0}}"] & U_{Nmn_{0}}\\[-1em]
\eta_{-Nm} \arrow[u,phantom, "\in" , sloped] \arrow[r, mapsto,shorten <= 1em,shorten >= 1em]& \eta_{-Nm+1} \arrow[u,phantom, "\in" , sloped] \arrow[r, mapsto,shorten >= 2em]& \cdots \arrow[r, mapsto,shorten <= 2em]& \eta_{-1} \arrow[u,phantom, "\in" , sloped] \arrow[r, mapsto,shorten <= 1em,shorten >= 1em]& \eta_{0} \arrow[u,phantom, "\in" , sloped]\\[-1em]
\eta' \arrow[u,phantom, "=" , sloped]& f^{n_{0}}(\eta') \arrow[u,phantom, "=" , sloped] && f^{(Nm - 1)n_{0}}(\eta')\arrow[u,phantom, "=" , sloped] &.
\end{tikzcd}
}
\end{equation}
Since $f^{-1}( \Omega \setminus \Sigma) \subset \Omega \setminus \Sigma$,
we have $\eta_{-i} \in  \Omega \setminus \Sigma$ for $i=0,\dots, Nm$.
Also, since the morphisms appeared above are all finite,
we have $\dim \overline{\{ \eta_{-i}\}} = \dim \overline{\{ \eta_{0}\}} = q' \geq q$.
Thus we have
\begin{align}
\sharp \{i \in \{0, 1, \dots, Nm\} \mid \eta_{-i} \in \Sigma''\} \leq l.
\end{align}

\begin{claim}\label{claim:ketabound}
We have
\begin{align}
\k_{Nmn_{0}}(\eta_{-Nm}) < \d^{Nmn_{0}}.
\end{align}
\end{claim}
Once we have proven this, we get 
\begin{align}
\k_{-n_{1}}(\eta_{0}) = \k_{-Nmn_{0}}(\eta_{0}) = \max_{\eta' \in f^{-Nmn_{0}}(\eta_{0})} \k_{Nmn_{0}}(\eta') < \d^{Nmn_{0}} = \d^{n_{1}}
\end{align}
and we are done.
\begin{proof}[Proof of \cref{claim:ketabound}]\ 

{\bf Case 1.} Suppose $\eta_{0} \notin \Sigma'$. (Note that this holds automatically when $q' > q$.)
In this case $\eta_{-i} \notin \Sigma'$ for all $i = 0,\dots, Nm$, since otherwise $\eta_{-i}$ coincide with
a generic point of $ \Sigma'$ and thus $f$-periodic. Thus $\eta_{0} \in \Sigma'$ and contradiction.
Thus $\eta_{-i} \in \Omega \setminus \widetilde{\Sigma}$ for all but at most $l$ $i$'s.
Thus by \cref{claim:delta0},
\begin{align}
\k_{Nmn_{0}}(\eta_{-Nm}) &\leq \k_{n_{0}}(\eta_{-Nm}) \k_{n_{0}}(\eta_{-Nm+1}) \cdots \k_{n_{0}}(\eta_{-1})\\
&\leq M^{l} (\d_{0}^{n_{0}})^{Nm} < \d^{Nmn_{0}}
\end{align}
and we are done.

{\bf Case 2.} Suppose $q'=q$ and $\eta_{-Nm} \in \Sigma'$. 
Then $\eta_{-i} \in \Sigma'$ for all $i=0,\dots, Nm$.
Since $\dim \overline{\{\eta_{-i}\}} = q$, $\eta_{-i}$ are generic points of $ \Sigma'$.
Thus 
\begin{align}
\k_{Nmn_{0}}(\eta_{-Nm}) &\leq \k_{mn_{0}}(\eta_{-Nm}) \k_{mn_{0}}(\eta_{-Nm+m}) \cdots \k_{mn_{0}}(\eta_{-m})\\
&\leq (\d_{1}^{mn_{0}})^{N} < \d^{Nmn_{0}}
\end{align}
and we are done.

{\bf Case 3.} Suppose $q'=q$, $\eta_{-Nm} \notin \Sigma'$, and $\eta_{0} \in \Sigma'$.
Let $s \in \{0,\dots, N\}$ be the largest such that $\eta_{-sm} \in \Sigma'$.
Then we have:

\begin{tikzpicture}
    \node (A) at (0,0) {$\eta_{-Nm}$};
    \node (B) at (2,0) {$\cdots$};
    \node (C) at (4,0) {$\eta_{-(s+1)m}$};
    \node (D) at (6,0) {$\cdots$};
    \node (E) at (8,0) {$\eta_{-sm}$};
    \node (F) at (10,0) {$\cdots$};
    \node (G) at (12,0) {$\eta_{0}$};

    \draw[|->] (A) -- (B) node[midway, above] {\scriptsize $f^{n_{0}}$};
    \draw[|->] (B) -- (C) node[midway, above] {\scriptsize $f^{n_{0}}$};
    \draw[|->] (C) -- (D) node[midway, above] {\scriptsize $f^{n_{0}}$};
    \draw[|->] (D) -- (E) node[midway, above] {\scriptsize $f^{n_{0}}$};
    \draw[|->] (E) -- (F) node[midway, above] {\scriptsize $f^{n_{0}}$};
    \draw[|->] (F) -- (G) node[midway, above] {\scriptsize $f^{n_{0}}$};
    
    \draw [decorate,decoration={brace,amplitude=10pt,mirror,raise=5pt}] (A.south west) -- (C.south east) node[midway,below=15pt] {\scriptsize $\notin \Sigma'$};
    \draw [decorate,decoration={brace,amplitude=10pt,mirror,raise=5pt}] (E.south west) -- (G.south east) node[midway,below=15pt] {\scriptsize generic point of $\Sigma'$};
\end{tikzpicture}

Thus
\begin{align}
&\k_{Nmn_{0}}(\eta_{-Nm}) \\
&\leq
 \prod_{n=0}^{Nm - (s+1)m -1}\k_{n_{0}}(\eta_{-Nm + n})
 \prod_{n=Nm - (s+1)m}^{Nm -sm -1}\k_{n_{0}}(\eta_{-Nm+n})
 \prod_{i=0}^{s-1}\k_{mn_{0}}(\eta_{-sm+im})\\
 &\leq 
 M^{l} (\d_{0}^{n_{0}})^{Nm - (s+1)m} M^{m}( \d_{1}^{mn_{0}})^{s}
 \leq M^{l+m} \max\{\d_{0},\d_{1}\}^{Nmn_{0}} < \d^{Nmn_{0}}
\end{align}
and we are done.
\end{proof}
This finishes the proof of \cref{prop:k-induction}.
\end{proof}

\begin{proof}[Proof of \cref{thm:kappanegative}]
Let $\dim X = k$.
Fix arbitrary $\d \in \R_{>1}$.
Let $n_{0} = 1$, $ \Omega = \Omega_{0} := U_{\infty}$, $q = k-1$ in \cref{prop:k-induction}.
Note that 
$\dim \overline{ \{ x \in U_{\infty} \mid \k_{-1}(x) \geq \d \} }\leq k-1$
because $\k_{-1} = 1$ at the generic point of $U_{\infty}$.
Then by \cref{prop:k-induction}, there are 
proper closed subset $ \Sigma_{1} \subsetneq U_{\infty}$ and $n_{1} \in \Z_{\geq1}$ such that
\begin{enumerate}
\item
generic points of  $\Sigma_{1}$ are union of $f$-periodic cycles;
\item
either $ \Sigma_{1} =  \emptyset$, or $ \overline{ \Sigma_{1}}$ is of pure dimension $k-1$;
\item
$\liminf_{n \to \infty} \k_{-n}(x)^{ \frac{1}{n}} \geq \d$ for all $x \in \Sigma_{1}$;
\item
$\dim \overline{ \{ x \in \Omega \setminus \Sigma_{1} \mid \k_{-n_{1}}(x) \geq \d^{n_{1}} \} } \leq k-2$.
\end{enumerate}
Moreover if we set $ \Omega_{1} = \Omega_{0} \setminus \Sigma_{1}$, then 
$\Omega_{1} \subset U_{\infty}$ is non-empty open and
$f^{-1}( \Omega_{1}) \subset \Omega_{1}$.
Therefore we can apply \cref{prop:k-induction} again.
Repeating this process and we get
proper closed subsets 
$ \Sigma_{1}, \dots , \Sigma_{k} \subsetneq U_{\infty}$ 
and $n_{1},\dots, n_{k} \in \Z_{\geq 1}$ with the following properties.
For each $i=1,\dots, k$, we set $ \Omega_{i} = U_{\infty} \setminus \bigcup_{j=1}^{i} \Sigma_{j}$.
Then
\begin{enumerate}
\item
generic points of  $\Sigma_{i}$ are union of $f$-periodic cycles;
\item
either $ \Sigma_{i} =  \emptyset$, or $ \overline{ \Sigma_{i}}$ is of pure dimension $k-i$;
\item
$\liminf_{n \to \infty} \k_{-n}(x)^{ \frac{1}{n}} \geq \d$ for all $x \in \Sigma_{i}$;
\item
$\dim \overline{ \{ x \in \Omega_{i-1} \setminus \Sigma_{i} \mid \k_{-n_{i}}(x) \geq \d^{n_{i}} \} } \leq k-i-1$.
\end{enumerate}
Moreover, we have $f^{-1}( \Omega_{i}) \subset \Omega_{i}$.
In particular, we have $\k_{-n_{k}} < \d^{n_{k}}$ on $ \Omega_{k}$.
Since $\k_{-n_{k}}$ is upper semicontinuous and $\Omega_{k}$ is Noetherian,
there is $\d_{0} < \d$ such that $\k_{-n_{k}} < \d_{0}^{n_{k}}$ on $ \Omega_{k}$.
Then for any $x \in \Omega_{k}$ and $n \geq 1$, we have
\begin{align}
\k_{-nn_{k}}(x) &= \max_{y \in f^{-nn_{k}}(x)} \k_{nn_{k}}(y) 
\leq \max_{y \in f^{-nn_{k}}(x)}  \prod_{i=0}^{n-1} \k_{n_{k}}(f^{n_{k}i}(y)) \\
&\leq  \max_{y \in f^{-nn_{k}}(x)}  \prod_{i=0}^{n-1} \k_{-n_{k}}(f^{n_{k}(i+1)}(y)) < (\d_{0}^{n_{k}})^{n}.
\end{align}
By \cref{lem:limsupinf-divisible}, we have
\begin{align}
\limsup_{n \to \infty} \k_{-n}(x)^{ \frac{1}{n}} = \limsup_{n \to \infty} \k_{-nn_{k}}(x)^{ \frac{1}{nn_{k}}} \leq \d_{0} < \d
\end{align}
for all $x \in \Omega_{k}$.

Now we take arbitrary $x \in U_{\infty}$.
Then define $\k_{-}(x) = \liminf_{n \to \infty} \k_{-n}(x)^{ \frac{1}{n}}$.
Apply the above argument to any $\d > \k_{-}(x)$.
Then we get $x \in \Omega_{k}$ and thus $\limsup_{n \to \infty} \k_{-n}(x)^{ \frac{1}{n}} < \d$.
Therefore $\lim_{n \to \infty}\k_{-n}(x)^{ \frac{1}{n}} $ exists and is equal to $\k_{-}(x)$.

Finally for arbitrary $\d \in \R_{>1}$, by the above argument again, we have
\begin{align}
\{ x \in U_{\infty} \mid \k_{-}(x) \geq \d \} = \Sigma_{1} \cup \cdots \cup \Sigma_{k}
\end{align}
and we are done.
\end{proof}

\begin{remark}
By the proof, for any $\d \in \R$, 
the generic points of the closed set
\begin{align}
\{x \in U_{\infty} \mid \k_{-}(x) \geq \d\}
\end{align}
are finite union of $f$-periodic cycles.
\end{remark}

\section{Multiplicities}\label{sec:mult}
Let $K$ be a field of characteristic zero.  

\begin{definition}
 Let $f \colon X \longrightarrow Y$ be a quasi-finite morphism of algebraic schemes over $K$.
 For a scheme point $x \in X$, we define 
 \begin{align}
 e_{f}(x) := l_{\O_{X,x}}(\O_{X,x} / f^{*}\m_{f(x)}\O_{X,x}).
 \end{align}
 By definition we have $ e_{f}(x) \geq 1$. If $f$ is unramified at $x$, then $ e_{f}(x)=1$. 
\end{definition}

\begin{lemma}\label{lem:uscofef}
Let $f \colon X \longrightarrow Y$ be a quasi-finite morphism of algebraic schemes over $K$.
Then the function
\begin{align}
e_{f} \colon X \longrightarrow \R
\end{align}
is upper semicontinuous.
\end{lemma}
\begin{proof}
This follows from, for example,
\cite[Corollary 4.8]{lejeune1974normal}
and the fact that there is a uniform $\nu \geq 1$ such that
$\m_{x}^{\nu} \O_{X,x} / f^{*}\m_{f(x)}\O_{X,x} =0$.
\end{proof}

\begin{lemma}\label{lem:submulef}
Let $f \colon X \longrightarrow Y$, $g \colon Y \longrightarrow Z$ be quasi-finite morphisms of algebraic schemes over $K$.
Then for any $x \in X$, we have
\begin{align}
e_{g\circ f}(x) \leq e_{f}(x) e_{g}(f(x)).
\end{align}
If $f$ is flat, the equality holds.
\end{lemma}
\begin{proof}
This follows from the following.
\begin{claim}\label{claim:chainrule}
Let $(A,\m_{A}) \longrightarrow (B,\m_{B}) \longrightarrow (C, \m_{C})$ be local homomorphisms between Noetherian local rings such that
$\m_{A}B$ is $\m_{B}$-primary and $\m_{B}C$ is $\m_{C}$ primary.
Then
\[ 
l_{C}(C/\m_{A}C) \leq l_{C}(C/\m_{B}C) l_{B}(B/\m_{A}B).
\]
If $B \longrightarrow C$ is flat, then
\[ 
l_{C}(C/\m_{A}C) = l_{C}(C/\m_{B}C) l_{B}(B/\m_{A}B).
\]
\end{claim}
\begin{proof}
Since $\m_{A}B$ is $\m_{B}$-primary and $\m_{B}C$ is $\m_{C}$-primary, 
$l_{B}(B/\m_{A}B)$ and $l_{C}(C/\m_{B}C)$ are finite.
By the exact sequence of $C$-modules
\[
0 \longrightarrow \m_{B}C/\m_{A}C \longrightarrow C/\m_{A}C \longrightarrow C/\m_{B}C \longrightarrow0 
\]
we have 
\[
l_{C}(C/\m_{A}C) = l_{C}(\m_{B}C/\m_{A}C) + l_{C}(C/\m_{B}C).
\]
Next consider the commutative digram
\begin{equation}
\begin{tikzcd}
&\m_{A}B {\otimes}_{B}C  \arrow[r] \arrow[d, ->>,"a"] & \m_{B} {\otimes}_{B} C \arrow[r] \arrow[d, ->>,"b"]& (\m_{B}/\m_{A}B) {\otimes}_{B}C \arrow[r] \arrow[d, ->>,"c"]& 0\\
0 \arrow[r] & \m_{A}C \arrow[r]& \m_{B}C \arrow[r] & \m_{B}C / \m_{A}C  \arrow[r] & 0 
\end{tikzcd}
\end{equation} 
where the rows are exact.
Note that $a,b,c$ are always surjective. If $B \longrightarrow C$ is flat, then $a,b$ are injective and hence by snake lemma
$c$ is also injective.
Thus 
\begin{align}
l_{C}(\m_{B}C / \m_{A}C) \leq l_{C}((\m_{B}/\m_{A}B) {\otimes}_{B}C)
\end{align}
and equality holds when $B \longrightarrow C$ is flat.
We claim that 
\begin{align}
l_{C}((\m_{B}/\m_{A}B) {\otimes}_{B}C) \leq l_{C}(C/\m_{B}C) l_{B}(\m_{B}/\m_{A}B)
\end{align}
and equality holds when $B \longrightarrow C$ is flat.
Indeed, let $M = \m_{B}/\m_{A}B$.
Let $M=M_{0} \supset M_{1} \supset \cdots \supset M_{l} =0$ be a composition sequence as $B$-module.
Then $M_{i}/M_{i+1} \simeq B/\m_{B}$ as $B$-modules.
Let us write $N = M {\otimes}_{B}C$.
Let $N_{i}$ be the image of $M_{i} {\otimes}_{B}C \longrightarrow N$.
Then $N=N_{0} \supset N_{1} \supset \cdots \supset N_{l} =0$ and thus
\begin{align}
l_{C}(N) = \sum_{i=0}^{l-1} l_{C}(N_{i}/N_{i+1}).
\end{align}
We have the following commutative diagram
\begin{equation}
\begin{tikzcd}
&M_{i+1} {\otimes}_{B}C  \arrow[r] \arrow[d, ->>,"a'"] & M_{i} {\otimes}_{B} C \arrow[r] \arrow[d, ->>,"b'"]& (M_{i}/M_{i+1}) {\otimes}_{B}C \arrow[r] \arrow[d, ->>,"c'"]& 0\\
0 \arrow[r] & N_{i+1} \arrow[r]& N_{i} \arrow[r] & N_{i}/N_{i+1} \arrow[r] & 0 
\end{tikzcd}
\end{equation} 
where the rows are exact.
Note that when $B \longrightarrow C$ is flat, $a', b'$ are injective and hence so is $c'$.
Thus 
\begin{align}
l_{C}(N_{i}/N_{i+1}) \leq l_{C}((M_{i}/M_{i+1}) {\otimes}_{B}C) = l_{C}((B/\m_{B}) {\otimes}_{B}C) = l_{C}(C/\m_{B}C)
\end{align}
with equality being true when $B \longrightarrow C$ is flat.
Thus we have proven
\begin{align}
l_{C}(N) \leq l_{B}(M) l_{C}(C/\m_{B}C)
\end{align}
and equality holds when $B \longrightarrow C$ is flat.

Therefore 
\begin{align}
l_{C}(C/\m_{A}C) \leq l_{C}(C/\m_{B}C) (l_{B}(\m_{B}/\m_{A}B) + 1) = l_{C}(C/\m_{B}C) l_{B}(B/\m_{A}B)
\end{align}
and equality holds when $B \longrightarrow C$ is flat.
\end{proof}

Applying \cref{lem:submulef} to the local homomorphisms
$\O_{Z, g(f(x))} \longrightarrow \O_{Y, f(x)} \longrightarrow \O_{X, x}$,
we are done.
\end{proof}

\begin{lemma}\label{lem:fldextandef}
Let $f \colon X \longrightarrow Y$ be a quasi-finite morphism of algebraic schemes over $K$.
Let $K \subset L$ be an algebraic extension.
Let $x' \in X_{L}$ and let $x$ be its projection to $X$.
Then
\begin{align}
e_{f_{L}}(x') = e_{f}(x)
\end{align}
\end{lemma}
\begin{proof}
Let us consider the following diagram.
\begin{equation}
\begin{tikzcd}
X_{L} \arrow[r,"f_{L}"] \arrow[d, "p", swap] & Y_{L} \arrow[d, "q"] \\
X \arrow[r,"f",swap] & Y
\end{tikzcd}
\end{equation} 
Let $y'=f_{L}(x')$ and $y = f(x) = q(y')$.
Since $K \subset L$ is algebraic extension, $\dim q^{-1}(y)=0$.
Since $K \subset L$ is separable extension, $q^{-1}(y)$ is reduced.
Thus we have 
\begin{align}
\m_{Y, y} \O_{Y_{L}, y'}  = \m_{Y_{L}, y'}.
\end{align}
By the same argument, we also have
\begin{align}
\m_{X, x} \O_{X_{L}, x'}  = \m_{X_{L}, x'}.
\end{align}
Therefore, 
\begin{align}
e_{f_{L}}(x') &= l_{\O_{X_{L}, x'}} (\O_{X_{L},x'} / \m_{Y_{L},y'}\O_{X_{L},x'})\\
& =  l_{\O_{X_{L}, x'}} (\O_{X_{L},x'} / \m_{Y,y}\O_{X_{L},x'})\\
& = l_{\O_{X_{L}, x'}} (\O_{X_{L},x'} / \m_{X,x}\O_{X_{L},x'}) l_{\O_{X,x}}(\O_{X,x} / \m_{Y,y}\O_{X,x})\\
&=e_{f}(x)
\end{align}
where for the third equality, we use \cref{claim:chainrule} and the flatness of $p$.
\end{proof}

Now let $X$ be a geometrically integral projective variety over $K$.
Let $f \colon X \dashrightarrow X$ be a dominant rational map.
Let $U_{n}, U_{\infty} = X_{f}^{\rm back}$ be as in \cref{section:uscbackconvergence}.
Then the functions 
\begin{align}
e_{f^{n}} \colon f^{-n}(U_{n}) \longrightarrow \R, \quad n \geq 0
\end{align}
form submultiplicative cocycle by \cref{lem:uscofef,lem:submulef}.
Therefore by \cref{thm:kappanegative}, we have
\begin{theorem}
In the above notation, the limit
\begin{align}
\lim_{n \to \infty} \max_{y \in f^{-n}(x)}e_{f^{n}}(x)^{ \frac{1}{n}} =: e_{f, -}(x)
\end{align}
exists for all $x \in X_{f}^{\rm back}$.
Moreover, the function
\begin{align}
e_{f,-} \colon X_{f}^{\rm back} \longrightarrow \R, x \mapsto e_{f,-}(x)
\end{align}
is upper semicontinuous and $e_{f,-}=1$ at the generic point of $X_{f}^{\rm back}$.
\end{theorem}
By the definition and the convergence, we have
\begin{align}
e_{f^{m}, - }(x) = e_{f, -}(x)^{m}
\end{align}
for $x \in X_{f}^{\rm back}( \subset X_{f^{m}}^{\rm back})$ and $m \geq 1$.

For any finite subset $Y \subset X_{f}^{\rm back}$, we define 
\begin{align}
e(f; Y) := \lim_{n \to \infty} \big(  \max_{y \in f^{-n}(Y)} e_{f^{n}}(y) \big)^{ \frac{1}{n}}.
\end{align}
The limit in the right hand side exists since it is maximum of finitely many convergent sequences:
\begin{align}
\lim_{n \to \infty} \big(  \max_{y \in f^{-n}(Y)} e_{f^{n}}(y) \big)^{ \frac{1}{n}} = \max_{x \in Y} e_{f,-}(x).
\end{align}
By this equality, for any $m \geq 1$, we have
\begin{align}
e(f^{m}, Y) = e(f,Y)^{m}.
\end{align}

Let $K \subset L$ be an algebraic extension.
Let $\pi \colon X_{L} \longrightarrow X$ be the projection.
Then for any $n \geq 0$, we have
\begin{equation}
\begin{tikzcd}
f_{L}^{-n}(\pi^{-1}(U_{n})) \arrow[r, "f_{L}^{n}"] \arrow[d,"\pi",swap] & \pi^{-1}(U_{n}) \arrow[d, "\pi"]\\
f^{-n}(U_{n}) \arrow[r, "f^{n}"] & U_{n}
\end{tikzcd}
\end{equation} 
where the vertical arrows are surjective.
By this, we have $\pi^{-1}(X_{f}^{\rm back}) \subset (X_{L})_{f_{L}}^{\rm back}$.
Moreover, by \cref{lem:fldextandef}, we have
\begin{align}
e(f_{L};\pi^{-1}(Y)) &= \lim_{n \to \infty} \big( \max_{y' \in f_{L}^{-n}(\pi^{-1}(Y))} e_{f_{L}^{n}}(y') \big)^{ \frac{1}{n}}\\
&=\lim_{n \to \infty} \big( \max_{y \in f^{-n}(Y)} e_{f^{n}}(y) \big)^{ \frac{1}{n}} = e(f;Y).
\end{align}

\bibliographystyle{acm}
\bibliography{arithmetic_degree}

\end{document}